\documentclass{amsart}
\newcommand{\RN}[1]{%
  \textup{\uppercase\expandafter{\romannumeral#1}}%
}

\usepackage{amsmath}
\usepackage{amssymb}
\usepackage{enumerate}
\usepackage[hidelinks]{hyperref}
\usepackage{mathtools}
\usepackage[capitalise]{cleveref}
\numberwithin{equation}{section}

\crefname{equation}{}{}
\newtheorem{theorem}{Theorem}[section]
\newtheorem{lemma}[theorem]{Lemma}
\newtheorem{proposition}[theorem]{Proposition}
\newtheorem{corollary}[theorem]{Corollary}

\theoremstyle{definition}
\newtheorem{definition}[theorem]{Definition}
\newtheorem{example}[theorem]{Example}

\theoremstyle{remark}
\newtheorem{remark}[theorem]{Remark}

\usepackage{enumerate}
\usepackage{tikz}
\usepackage{tikz-cd}
\usetikzlibrary{shapes,arrows}
\tikzstyle{line}=[draw] 
\usepackage{enumitem}
\usepackage{epigraph}

\author{Yu Zhao}
\title{A Categorical Quantum Toroidal Action On The Hilbert Schemes}
\date{\today}

\begin{document}
\maketitle
\begin{center}
	\small{\textit{To Andrei, Lucas and Linyuan with gratitude}}
\end{center}
\begin{abstract}
  We categorify the commutation of Nakajima's Heisenberg operators $P_{\pm 1}$ and their infinitely many counterparts in the quantum toroidal algebra $U_{q_1,q_2}(\ddot{gl_1})$ acting on the Grothendieck groups of Hilbert schemes from \cite{MR2854154,schiffmann2013,negut2017shuffle,neguct2018hecke}. By combining our result with \cite{neguct2018hecke}, one obtains a geometric categorical $U_{q_1,q_2}(\ddot{gl_1})$ action on the derived category of Hilbert schemes. Our main technical tool is a detailed geometric study of certain nested Hilbert schemes of triples and quadruples, through the lens of the minimal model program, by showing that these nested Hilbert schemes are either canonical or semi-divisorial log terminal singularities.
\end{abstract}

\section{Introduction}

\subsection{The Description of Our Result}

The quantum toroidal algebra $U_{q_1,q_2}(\ddot{gl}_1)$ (\cref{def:2.1}) is an affinization of the quantum Heisenberg algebra, studied by Ding-Iohara-Miki \cite{MR1463869,MR2377852}, and has been realized in several contexts:
\begin{itemize}
\item the elliptic Hall algebra in \cite{MR2922373,MR2886289},
\item the double shuffle algebra in \cite{MR2566895,feigin2001quantized,MR3283004},
\item the trace of the deformed Khovanov Heisenberg category in \cite{MR3874690}(when $q_1=q_2$).
\end{itemize}

Given a smooth quasi-projective surface $S$ over $k=\mathbb{C}$, let
$$\mathcal{M}=\bigsqcup_{n=0}^\infty S^{[n]}$$
be the Hilbert schemes of points on $S$. Schiffmann-Vasserot \cite{schiffmann2013}, Feigin-Tsymbaliuk \cite{MR2854154} and Negu{\c t} \cite{neguct2018hecke} constructed the $U_{q_1,q_2}(\ddot{gl}_1)$ action on the Grothendieck group of $\mathcal{M}$. It generalizes the action of
\begin{itemize}
\item the Heisenberg algebra (Nakajima \cite{MR1711344} and Grojnowski \cite{grojnowski1996instantons})
\item the $W$ algebra (Li-Qin-Wang \cite{10.1155/S1073792802110129})
\end{itemize}
 on the cohomology of Hilbert schemes. There are already some applications of this action in algebraic geometry, like the Beauville-Voison conjecture for the Hilbert schemes of points on $K3$ surfaces \cite{maulik2020}.

The purpose of this paper is to categorify the above quantum toroidal algebra action. We prove that
\begin{theorem}[\cref{thm:7.1}]
  \label{thm:1.1}
  Consider the correspondences $e_k,f_k:D^b(\mathcal{M})\to D^b(\mathcal{M}\times S)$ induced from:
  \begin{align*}
    e_{k}:=\mathcal{L}^{k}\mathcal{O}_{S^{[n,n+1]}}\in D^b(S^{[n]}\times S^{[n+1]}\times S) \\
    f_k:=\mathcal{L}^{k-1}\mathcal{O}_{S^{[n,n+1]}}[1]\in D^b(S^{[n+1]}\times S^{[n]}\times S).
  \end{align*}
  where
  $S^{[n,n+1]}$ is the nested Hilbert scheme
$$S^{[n,n+1]}:=\{(\mathcal{I}_n,\mathcal{I}_{n+1},x)\in S^{[n]}\times S^{[n+1]}\times S|\mathcal{I}_{n+1}\subset \mathcal{I}_n, \mathcal{I}_n/\mathcal{I}_{n+1}=k_x\}$$
and the line bundle $\mathcal{L}$ on $S^{[n,n+1]}$ has fibers equal to $\mathcal{I}_{n}/\mathcal{I}_{n+1}$. Then
\begin{enumerate}
\item  For every two integers $m$ and $r$, there exists natural transformations
  \begin{equation}
    \label{1cone}
    \begin{cases}
      f_re_{m-r}\to e_{m-r}f_r & \text{ if } m>0 \\
      e_{m-r}f_r\to f_re_{m-r} & \text{ if } m<0 \\
      f_re_{-r}=e_rf_{-r}\oplus \mathcal{O}_{\Delta}[1]
    \end{cases}
  \end{equation}
   where $\Delta$ is the diagonal of $\mathcal{M}\times \mathcal{M}\times S \times S$.
 \item  When $m\neq 0$, the cone of the natural transformations in \cref{1cone} has a filtration with associated graded object
   \begin{equation*}
     \begin{cases}
       \bigoplus\limits_{k=0}^{m-1}R\Delta_*(h_{m,k}^{+}) & \text{ if } m>0 \\
       \bigoplus\limits_{k=m+1}^{0}R\Delta_*(h_{m,k}^{-}) & \text{ if } m<0
     \end{cases}
   \end{equation*}
   where $h_{m,k}^{+},h_{m,k}^{-}\in D^{b}(\mathcal{M}\times S)$ are defined in \cref{hmk} and are complexes of universal sheaves on $\mathcal{M}\times S$.

 \item (\cref{thm:4.9})  At the level of Grothendieck groups, we have the formula:
    \begin{align*}
      (1-[\omega_S])\sum_{k=0}^{m-1}[h_{m,k}^+]=h_m^{+} & \qquad m>0 \\
      (1-[\omega_S])\sum_{k=m+1}^{0}[h_{m,k}^-]=h_{-m}^{-} & \qquad m<0
    \end{align*}
   where $\omega_S$ is the canonical line bundle of $S$ and $h_m^\pm$ is defined in \cite{negut2017shuffle} (also see \cref{4eq:lem2.11}) .
 \end{enumerate}
  Note that the non-triviality of the extension is a feature of the derived category statement, which is not visible at the level of Grothendieck groups. \cref{ext2} provides a precise extension formula.
\end{theorem}

The positive part of the quantum toroidal algebra action was already categorified in \cite{neguct2018hecke}. \cref{thm:1.1} categorifies the commutation of the positive part and the negative part, and thus accounts for the action of the whole quantum toroidal algebra.
\begin{remark}
  We have not categorified all the relations defining the elliptic Hall algebra of \cite{MR2922373} (see \cite{neguct2018hecke} for more discussions).
\end{remark}

\begin{remark}
 Our categorification is a geometric categorification in the sense of \cite{MR2668555}. The $2$-categorification of the quantum toroidal algebra is still unclear to us.
\end{remark}
\subsection{Outline of The Proof}
Let us first review the categorification of the commutation of $e_k$ and $e_l$ for different $k$ and $l$ in \cite{neguct2018hecke}. Consider the moduli space $\mathfrak{Z}_2,\mathfrak{Z}_{2}'$ which parameterize diagrams

\begin{minipage}{0.5\linewidth}
  \begin{equation}
  \begin{tikzcd}
    \mathcal{I}_{n-1} \ar[r,hook,"x"] & \mathcal{I}_{n}  \ar[r,hook,"y"] &  \mathcal{I}_{n+1}
  \end{tikzcd}
\end{equation}
\end{minipage}
\begin{minipage}{0.5\linewidth}
  \begin{equation}
  \begin{tikzcd}
    \mathcal{I}_{n-1}  \ar[r,hook,"y"] & \mathcal{I}_{n}'  \ar[r,hook,"x"] &  \mathcal{I}_{n+1}
  \end{tikzcd}
\end{equation}
\end{minipage}
respectively, of ideal sheaves where each successive inclusion is colength $1$ and supported at the point indicated on the diagrams.  Then $e_ke_l$ and $e_le_k$ are the derived pushfoward of line bundles on $\mathfrak{Z}_2$ and $\mathfrak{Z}_2'$ to $S^{[n-1]}\times S^{[n+1]}\times S\times S$ respectively.

In order to compare $e_ke_l$ and $e_le_k$, \cite{neguct2018hecke} introduced the quadruple moduli space $\mathfrak{Y}$ which parameterize diagrams
\begin{equation}
  \begin{tikzcd}
    &\mathcal{I}_n  \ar[rd,hook,"y"] & \\
    \mathcal{I}_{n+1} \ar[ru,hook,"x"] \ar[rd,hook,swap,"y"] & & \mathcal{I}_{n-1}\\
    & \mathcal{I}_n' \ar[ru,hook,"x",swap] & 
  \end{tikzcd}
\end{equation}
of ideal sheaves where each successive inclusion is colength $1$ and supported at the point indicated on the diagrams. $\mathfrak{Z}_{+}$ is smooth and induces resolutions of $\mathfrak{Z}_{2}$ and $\mathfrak{Z}_{2}'$. Proposition 2.30 of \cite{neguct2018hecke} proved that $\mathfrak{Z}_{2}$ and $\mathfrak{Z}_{2}'$ are rational singularities, based on the fact that any fiber of the resolution has dimension $\leq 1$. Thus $e_ke_l$ and $e_le_k$ could be compared through line bundles on $\mathfrak{Y}$.

Now in order to compare $f_re_{m-r}$ and $e_{m-r}f_r$, we introduce the triple moduli spaces $\mathfrak{Z}_+,\mathfrak{Z}_{-}$ which parameterize diagrams

\begin{minipage}{0.5\linewidth}
  \begin{equation}
  \begin{tikzcd}
    &\mathcal{I}_n \\
    \mathcal{I}_{n+1} \ar[ru,hook,"x"] \ar[rd,hook,"y",swap] & \\
    & \mathcal{I}_n'
  \end{tikzcd}
\end{equation}
\end{minipage}
\begin{minipage}{0.5\linewidth}
  \begin{equation}
  \begin{tikzcd}
    \mathcal{I}_{n} \ar[rd,hook,"y"] & \\
    & \mathcal{I}_{n-1} \\
    \mathcal{I}_n' \ar[ru,hook,"x",swap] &
  \end{tikzcd}
\end{equation}
\end{minipage}
Then $f_re_{m-r}$ and $e_{m-r}f_r$ are the derived push forward of line bundles on $\mathfrak{Z}_+,\mathfrak{Z}_{-}$ respectively. The quadruple moduli space $\mathfrak{Y}$ still induces resolutions of $\mathfrak{Z}_{-}$, but $\mathfrak{Z}_{+}$ have two irreducible components. One irreducible component is $S^{[n,n+1]}$, denoted by $W_1$ and the other irreducible component is denoted by $W_0$  in \cref{sec:7}. $\mathfrak{Y}$ induces a resolution of $W_0$.

In order to compare $f_re_{m-r}$ and $e_{m-r}f_r$ through line bundles on $\mathfrak{Y}$, one must prove that $\mathfrak{Z}_{-}$ and $W_0$ are rational singularities. The approach in \cite{neguct2018hecke} did not work here, as the fiber could have pretty large dimensions. Instead, we study the singularity structure of $\mathfrak{Z}_{+}$ and $\mathfrak{Z}_-$ through the viewpoint of the minimal model program (MMP) \cite{Kollar-Mori,MR3057950}. We prove that
\begin{proposition}[\cref{1pr:3.1} and \cref{besemidlt}]
  \label{above}
  The pair $(\mathfrak{Z}_+,0)$ is semi-dlt. $\mathcal{Z}_{-}$ and $W_0$ are canonical singularities.
\end{proposition}
We prove \cref{above} by explicitly computing the discrepancy (see \cref{sec:7} for the definitions of semi-dlt, canonical singularities and the discrepancy). Canonical singularities are always rational singularities by \cref{thm:5.22}.

\subsection{Higher Rank Stable Sheaves and the Deformed $\mathcal{W}$-algebra}
The quantum toroidal algebra also acts on the Grothendieck group of higher rank stable sheaves \cite{neguct2018hecke}. The action factors through the deformed $W$-algebra \cite{neguct2017w}, which leads to the AGT correspondences for algebraic surfaces.

We will pursue the categorification of the above algebra action in the future. Different to the case of Hilbert schemes, $\mathfrak{Z}_+$ is no longer equi-dimensional and the obstruction bundle has to be accounted for.

\subsection{Categorical Heisenberg Actions}

Khovanov \cite{MR3205569} defined the Heisenberg category through graphical calculus. Cautis-Licata \cite{MR2988902} constructed a categorical Heisenberg action on the derived category of Hilbert schemes of points of the minimal resolution of the type ADE singularities. Krug \cite{MR3774400} constructed the weak categorical Heisenberg action on the derived category of Hilbert schemes of points on smooth surfaces. Our categorification is different from those above, as it is given in terms of explicit correspondences and independent of the derived McKay correspondence.  


Although the higher Nakajima operators were categorified by the objects $e_{(0,...,0)}$ of \cite{neguct2018hecke}, the relations between them (as well as the morphisms between them in Khovanov's Heisenberg category) are still unclear to us.
\subsection{Double Categorified Hall Algebra} The study of Cohomological Hall algebra was initiated by Kontsevich-Soibelman \cite{kontsevich2010cohomological} and Schiffmann-Vasserot \cite{MR3150250}. Kapranov-Vasserot \cite{kapranov2019cohomological} and the author \cite{zhao2019k} constructed the $K$-theoretic Hall algebra on surfaces, which was categorified by Porta-Sala \cite{porta2019categorification}. It also categorified the positive half of $U_{q}(\ddot{gl}_1)$ when $S=\mathbb{A}^2$. The relation between the categorified Hall algebra of minimal resolution of type A singularities and quivers was studied by Diaconescu-Porta-Sala \cite{diaconescu2020mckay}. On the other hand, the Drinfeld double of the categorified Hall algebra is still mysterious.  As an attempt to understand the action of the ``double Categorified Hall algebra'', it is natural to expect that our approach could be generalized to categorifications in other settings, like those of Toda \cite{toda2020hall} and Rapcak-Soibelman-Yang-Zhao \cite{rapcak2020cohomological}.

\subsection{Other Related Work}

Recently, Addington-Takahashi \cite{takahashi2020categorical} studied certain sequences of moduli spaces of sheaves on $K3$ surfaces and show that these sequences can be given the structure of a geometric categorical $\mathfrak{sl}_2$ action in the sense of \cite{MR2668555}. It would be interesting to explore the interactions between their action and ours.

Another related work is Jiang-Leung's projectivization formula \cite{jiang2018derived}. Through this formula, they obtained a semi-orthogonal decomposition of the derived category of the nested Hilbert schemes.

\subsection{The Organization of The Paper} The proof of the main theorem is in \cref{sec:6} and the extension formula is in \cref{sec:7}. The other sections are organized as follows:
\begin{description}
\item[\cref{sec:2}]  we review the action of $U_{q_1,q_2}(\ddot{gl}_1)$ on the Grothendieck group of Hilbert schemes \cite{MR2854154,schiffmann2013,negut2017shuffle,neguct2018hecke};
\item[\cref{sec:3}] we define $h_{m,k}^{+}\in D^b(\mathcal{M}\times S)$ and prove the third part of \cref{thm:1.1};
\item[\cref{sec:4}] we make a review of the singularity of the minimal model program.
\item[\cref{sec:5}] we study the singularity structures of $\mathfrak{Z}_{-}$ and $\mathfrak{Z}_{+}$ through the singularity theory of the minimal model program.
\end{description}

\subsection{Acknowledgment}  The author is grateful to Jihao Liu and Ziquan Zhuang for their help in understanding the minimal model program and Yuchen Liu for checking the \cref{sec:4}. The author would like to thank Qingyuan Jiang, Huachen Chen, Xiaolei Zhao, Shizhuo Zhang, Andreas Krug and Sabin Cautis for many interesting discussions on the subject. The author is supported by the NSF Career grant DMS-1845034.

\section{The Quantum Toroidal Algebra $U_{q_1,q_2}(\ddot{gl}_1)$ and the $K$-theory of Hilbert scheme of points on surfaces}
\label{sec:2}
In this section, we will review the action of $U_{q_1,q_2}(\ddot{gl}_1)$ on the $K$-theory on Hilbert scheme of points on surfaces from \cite{MR2854154,schiffmann2013,negut2017shuffle,neguct2018hecke}. It will be formulated in \cref{2.17}.

In this paper, we will denote $K(X)$ the Grothendieck group of coherent sheaves on $X$ for a scheme $X$. 
 \subsection{Hilbert Schemes and Nested Hilbert Schemes}
 \label{sec:Hilbert}
 Given an integer $n>0$ and a smooth quasi-projective surface $S$ over $k=\mathbb{C}$, let
 $$S^{[n]}:=\{\mathcal{I}_n\subset \mathcal{O}|\mathcal{O}/\mathcal{I}_n \text{ is dimension } 0 \text{ and length } n\}$$
 be the Hilbert scheme of $n$ points on $S$. There is a universal ideal sheaf on $S^{[n]}\times S$, still denoted by $\mathcal{I}_n$, and a universal closed subscheme $\mathcal{Z}_n\subset S^{[n]}\times S$. 
'\begin{proposition}[Proposition 2.11 of \cite{neguct2018hecke}]
  \label{1lm:2.2}
  There exists a resolution of $\mathcal{I}_{n}$ by
  \begin{equation}
    \label{eq:resolve}
  0\to W_n \xrightarrow{s} V_n \to \mathcal{I}_{n}\to 0    
  \end{equation}
  where $W_n$ and $V_n$ are locally free coherent sheaves with the same determinant. Let $w_n$ and $v_n$ be the rank of $W_n$ and $V_n$, respectively. Then $v_n-w_n=1$.
\end{proposition}
\begin{definition}
  The nested Hilbert scheme $S^{[n,n+1]}$ is defined to be
$$S^{[n,n+1]}:=\{(\mathcal{I}_n,\mathcal{I}_{n+1},x)\in S^{[n]}\times S^{[n+1]}\times S|\mathcal{I}_{n+1}\subset \mathcal{I}_n, \mathcal{I}_n/\mathcal{I}_{n+1}=k_x\}$$
with natural projection maps

\begin{minipage}{0.5\linewidth}
  \begin{equation}
    \begin{tikzcd}
      & S^{[n,n+1]} \ar{rd}{p_n^{-}} \ar[ld,swap, "p_n^+"] \ar{d}{\pi_n} & \\
      S^{[n]} & S & S^{[n+1]}
    \end{tikzcd}
  \end{equation}
\end{minipage}
\begin{minipage}{0.5\linewidth}
  \begin{equation}
    \begin{tikzcd}
      & (\mathcal{I}_n,\mathcal{I}_{n+1},x) \ar{rd}{p_n^{-}} \ar[ld,swap, "p_n^+"] \ar{d}{\pi_n} & \\
      \mathcal{I}_n & x & \mathcal{I}_{n+1}
    \end{tikzcd}
  \end{equation}
\end{minipage}
and let
\begin{align*}
  p_n:=(p_{n}^+,\pi_n): S^{[n,n+1]}\to S^{[n]}\times S.
\end{align*}

The tautological line bundle $\mathcal{L}$ is the line bundle whose fibers are  $\mathcal{I}_{n}/\mathcal{I}_{n+1}$.
\end{definition}

\begin{definition}
  Define the nested Hilbert scheme $S^{[n-1,n,n+1]}$ by the Cartesian diagram:
  \begin{equation}
    \label{diagram4.1}
    \begin{tikzcd}
      S^{[n-1,n,n+1]} \ar{r}{q_n}\ar{d} & S^{[n,n+1]} \ar{d}{p_{n}} \\
      S^{[n-1,n]} \ar{r}{(p_{n-1}^{-},\pi_{n-1})} & S^{[n]}\times S
    \end{tikzcd}
  \end{equation}
which consists of
$$\{(\mathcal{I}_{n-1},\mathcal{I}_n,\mathcal{I}_{n+1},x)\in S^{[n-1]}\times S^{[n]}\times S^{[n+1]}\times S|\mathcal{I}_{n-1}/\mathcal{I}_{n}=k_x,\mathcal{I}_{n}/\mathcal{I}_{n+1}=k_x\}.$$
There are two tautological line bundles $\mathcal{L}_1,\mathcal{L}_{2}$ on $S^{[n-1,n,n+1]}$ whose fibers are $\mathcal{I}_{n}/\mathcal{I}_{n+1},\mathcal{I}_{n-1}/\mathcal{I}_n$ respectively.
We denote
$$q_n:S^{[n-1,n,n+1]}\to S^{[n,n+1]}$$
 the projection morphism.
\end{definition}
We define
$$\mathcal{M}:=\bigsqcup_{n=0}^\infty S^{[n]}$$
to be the Hilbert schemes of points on $S$. 

\begin{example}

  \label{example:2.5}
  Let $\Delta_{S}:S\to S\times S$ be the diagonal embedding and $\mathcal{I}_{\Delta_S}$ be the ideal sheaf of the diagonal. Then
  \begin{align*}
    S^{[1,2]}=Bl_{\Delta_S}(S\times S)=\mathbb{P}_{S\times S}(\mathcal{I}_{\Delta_{S}}) \qquad  S^{[2]}=Bl_{\Delta_{S}}(S\times S)/\mathbb{Z}_2
  \end{align*}
    where the $\mathbb{Z}_2$ action on $Bl_{S\times S}(\mathcal{I}_{\Delta_{S}})$ is induced by the $\mathbb{Z}_2$ action
  $$i:S\times S\to S\times S \quad i(x,y)=(y,x).$$
  By \cite{MR3484148}, the projection morphism 
  $$(p_{2}^-,\pi_2):S^{[1,2]}\to S^{[2]}\times S\quad (\mathcal{I}_1,\mathcal{I}_2,x)\to (\mathcal{I}_2,x)$$
  has image $\mathcal{Z}_2$ and is an isomorphism when restricting to $\mathcal{Z}_2$. By \cref{diagram4.1}, $S^{[1,2,3]}\cong p_2^{-1}(\mathcal{Z}_2)$.
\end{example}

\subsection{The Quantum Toroidal algebra $U_{q_1,q_2}(\ddot{gl}_1)$}
We follow the notation of \cite{neguct2018hecke} for the definition of the quantum toroidal algebra $U_{q_1,q_2}(\ddot{gl}_1)$. Given two formal parameters $q_1$ and $q_2$, let $q=q_1q_2$.  Let
  $$\mathbb{K}=\mathbb{Z}[q_1^{\pm 1},q_2^{\pm 1}]^{Sym}_{([1],[2],[3],\cdots)}$$
  where $Sym$ refers to polynomials which are symmetric in $q_1$ and $q_2$.
 
 \begin{definition}
   \label{def:2.1}
  The quantum toroidal algebra $U_{q_1,q_2}(\ddot{gl_1})$ is the $\mathbb{K}$-algebra with generators:
  $$\{E_{k},F_{k},H_{l}^\pm\}_{k\in \mathbb{Z},l\in \mathbb{N}}$$
  modulo the following relations:
  \begin{multline}
    \label{2.1}
    (z-wq_1)(z-wq_2)(z-\frac{w}{q})E(z)E(w)= \\
    =(z-\frac{w}{q_1})(z-\frac{w}{q_2})(z-wq)E(w)E(z)
  \end{multline}
  \begin{multline}
    \label{2.2}
    (z-wq_1)(z-wq_2)(z-\frac{w}{q})E(z)H^{\pm}(w)= \\
    =(z-\frac{w}{q_1})(z-\frac{w}{q_2})(z-wq)H^{\pm}(w)E(z)
  \end{multline}
  \begin{equation}
    \label{2.3}
     [[E_{k+1},E_{k-1}],E_k]=0 \quad \forall k\in \mathbb{Z}
   \end{equation}
   together with the opposite relations for $F(z)$ instead of $E(z)$, as well as:
   \begin{equation}
     \label{2.4}
     [E(z),F(w)]=\delta(\frac{z}{w})(1-q_1)(1-q_2)(\frac{H^+(z)-H^{-}(w)}{1-q}) 
   \end{equation}
   where
   \begin{equation}
     E(z)=\sum_{k\in \mathbb{Z}}\frac{E_k}{z^k}, \quad F(z)=\sum_{k\in \mathbb{Z}}\frac{F_k}{z^k}, \quad H^{\pm}(z)=\sum_{l\in \mathbb{N}\cup \{0\}}\frac{H_l^\pm}{z^{\pm l}}
   \end{equation}
   where  $$\delta(z)=\sum_{n\in \mathbb{Z}}z^n\in \mathbb{Q}\{\{z\}\}.$$ We will set $H_0^+=q$ and $H_0^{-}=1$.
 \end{definition}
 
\subsection{The Quantum Toroidal Algebra Action on the $K$-theory of Hilbert Schemes} We will write $S_1,S_2$ for two copies of $S$, in order to emphasize  the  factors of $S\times S$.
\begin{definition}[Definition 4.10 and Definition 4.11 of \cite{neguct2018hecke}]
  \label{def:4.10}
     For any group homomorphisms $x,y:K(\mathcal{M})\to K(\mathcal{M}\times S)$, we define:
     $$xy|_{\Delta_S}=\{K(\mathcal{M})\xrightarrow{y} K(\mathcal{M}\times S)\xrightarrow{x\times Id_{S}}K(\mathcal{M}\times S\times S)\xrightarrow{Id_{\mathcal{M}}\times \Delta_S^*}K(\mathcal{M}\times S)\}$$
     where $\Delta_S:S\to S\times S$ is the diagonal embedding. Also define:
\begin{multline*}
  [x,y]=\{K(\mathcal{M})\xrightarrow{y} K(\mathcal{M}\times S_2)\xrightarrow{s\times Id_{S_2}}K(\mathcal{M}\times S_1\times S_2)\} \\
  -\{K(\mathcal{M})\xrightarrow{x} K(\mathcal{M}\times S_1)\xrightarrow{y\times Id_{S_1}}K(\mathcal{M}\times S_1\times S_2)\}
\end{multline*}
We define
$$[x,y]_{red}=z$$
for a group homomorphism $z:K(\mathcal{M})\to K(\mathcal{M\times S})$ if
  \begin{equation*}
    [x,y]=\Delta_{S*}(z).
  \end{equation*}
  The definition is unambiguous, since $\Delta_{S*}:K(S)\to K(S\times S)$ is injective, and so $z$ is unique.
\end{definition}

\begin{definition}
  For a two term complex of locally free sheaves
  $$U:=\{W\xrightarrow{\mathfrak{u}}V\}$$
  we define
  $$[\wedge^k(U)]:=\sum_{i=0}^k(-1)^i[S^iW][\wedge^{k-i}V] \quad [S^k(U)]:=\sum_{i=0}^k(-1)^i[\wedge^iW][S^{k-i}V]$$
  as elements in $K(X)$ and define $det(U):=\frac{det(V)}{det(W)}$. We define $U^\vee$ to be the  two term complex
  $$\{V^\vee\xrightarrow{\mathfrak{u}^\vee}W^\vee\}.$$
\end{definition}
\begin{definition}
  \label{4eq:lem2.11}
  We define $h_0^{+}:=[\omega_S]$, $h_{0}^{-}:=1$ and when $m>0$
  \begin{align}
    \label{4eq:total}
      h_{m}^{+}:=(1-\omega_S)\sum_{j=0}^{m-1}[\omega_S^{-j}]\sum_{i=0}^{j}(-1)^{i}[S^{m-i}\mathcal{I}_n][\wedge^{i}\mathcal{I}_n] \\
       h_m^{-}:=(1-\omega_S)\sum_{j=0}^{m-1}(-1)^{j}[\omega_S^{j}]\sum_{i=0}^j(-1)^i[\wedge^{m-i}\mathcal{I}_n^\vee][S^i\mathcal{I}_n^\vee]
  \end{align}
  as elements in $K(S^{[n]}\times S).$ Here we abuse the notation to denote
  $$\mathcal{I}_n:=\{W_n\xrightarrow{s}V_n\}$$
  in the short exact sequence \cref{eq:resolve}.
\end{definition}
\begin{remark}
  \cref{4eq:lem2.11} is equivalent to the definition of $h_{m}^{\pm}$ in \cite{negut2017shuffle}. We will prove it in \cref{sec:b}.
\end{remark}

\begin{theorem}[Theorem 1.2 of \cite{negut2017shuffle}]
  \label{2.17}
  Let $T^*S$ be the cotangent bundle of $S$ and $\omega_S$ be the canonical bundle of $S$. The morphism:
$$q_1+q_2\to [T^*S] \quad q=q_1q_2\to [\omega_S]$$
induces a homomorphism:
$$\mathbb{K}\to K_S.$$
  Regarding
  \begin{align*}
  \tilde{e}_i:=[\mathcal{L}^i\mathcal{O}_{S^{[n,n+1]}}]\in K(S^{[n]}\times S^{[n+1]}\times S),\\
  \tilde{f}_i:=-[\mathcal{L}^{i-1}\mathcal{O}_{S^{[n,n+1]}}]\in K(S^{[n+1]}\times S^{[n]}\times S), \\
  \tilde{h}_{i}^\pm:=[h_{i}^\pm\mathcal{O}_{S^{[n]}\times S}]\in  K(S^{[n]}\times S^{[n]}\times S)
\end{align*}
 as operators $K(\mathcal{M})\to K(\mathcal{M}\times S)$ through the $K$-theoretic correspondences, where we regard $S^{[n]}\times S$ as a closed subscheme of $S^{[n]}\times S^{[n]}\times S$ through the diagonal embedding, then there exists a unique $\mathbb{K}$-homomorphism
  $$\Phi:U_{q_1,q_2}(\ddot{gl}_1)\to Hom(K_{\mathcal{M}},K_{\mathcal{M}\times S})$$
  such that
  \begin{enumerate}
  \item $$\Phi(E_i)=\tilde{e}_i, \quad \Phi(F_i)=\tilde{f}_i, \quad \Phi(H_i^{\pm})=\tilde{h}_i^\pm$$
  \item For all $x,y\in  U_{q_1,q_2}(\ddot{gl}_1)$, we have
    \begin{align*}
    \Phi(xy)=\Phi(x)\Phi(y)|_{\Delta_S} \\
    [\Phi(x),\Phi(y)]|_{red}=\Phi(\frac{[x,y]}{(1-q_1)(1-q_2)}).
    \end{align*}
    The right-hand side is well defined due to the fact that all commutators in $U_{q_1,q_2}(\ddot{gl}_1)$ are multiples of $(1-q_1)(1-q_2)$(see Theorem 2.4 of \cite{neguct2017w}).
  \end{enumerate}
\end{theorem}

\section{Nested Hilbert Schemes and $h_{m,k}^{\pm}$}
\label{sec:3}
Given any scheme $X$, let $D^b(X)$ be the bounded derived category of coherent sheaves on $X$. In this section, we will define objects
$$h_{m,k}^{\pm}\in D^b(S^{[n]}\times S)$$
such that at the level of Grothendieck groups
\begin{equation}
  \label{eq:gro}
  \begin{cases}
    (1-[\omega_S])\sum\limits_{k=0}^{m-1}[h_{m,k}^+]=h_m^{+} & \qquad m>0 \\
      (1-[\omega_S])\sum\limits_{k=m+1}^{0}[h_{m,k}^-]=h_{-m}^{-} & \qquad m<0.
  \end{cases}
\end{equation}

 \begin{definition}
    \label{hmk}
    For two integers $k,m$ such that $m>k\geq 0$, we define $h_{m,k}^{+}\in D^{b}(S^{[n]}\times S)$ by
    \begin{equation}
      h_{m,k}^+:=
      \begin{cases}
        R(p_n\circ q_n)_*(\mathcal{L}_1^{m-1-k}\mathcal{L}_2^k)[1] & k>0 \\
        Rp_{n*}(\mathcal{L}^{m})[2] & k=0
      \end{cases}
    \end{equation}
    and if $m<k \leq 0$, we define $h_{m,k}^-\in  D^{b}(S^{[n]}\times S)$ by
    \begin{equation}
      h_{m,k}^{-}:=
      \begin{cases}
        R(p_n\circ q_n)_*(\mathcal{L}_1^{m-1-k}\mathcal{L}_2^{k})[1] & k<0 \\
        Rp_{n*}(\mathcal{L}^{-m-1})[1] & k=0.
      \end{cases}
    \end{equation}
  \end{definition}
 \begin{theorem}
    \label{thm:4.9} At the level of Grothendieck groups,
    \begin{align*}
  [h_{m,k}^+]=[\omega_S^{-k}]\sum_{i=0}^{k}(-1)^{i}[S^{m-i}\mathcal{I}_n][\wedge^{i}\mathcal{I}_n] \\
  [h_{m,k}^-]=(-1)^{k}[\omega_S^{k}]\sum_{i=0}^k(-1)^i[\wedge^{m-i}\mathcal{I}_n^\vee][S^i\mathcal{I}_n^\vee] 
\end{align*}
and \cref{eq:gro} holds.
     \end{theorem}
\cref{thm:4.9} will be proved later in this section.

\subsection{Projectivization and  A Categorical Projection Lemma} In this paper, for any complex
$$\{\cdots \to C_{-1}\to C_0\to 0\}$$
we will assume that $C_0$ has cohomological degree $0$ unless explicitly pointing out the cohomological degree.
\begin{definition}
  Given a two term complex of locally free sheaves $U:=\{W\xrightarrow{s}V\}$, we define the symmetric product and the wedge product complexes: 
  \begin{align*}
  S^{k}(U):=\{\wedge^{k}W\cdots \to \cdots \to W\otimes S^{k-1}V\to S^k(V) \}\\
  \wedge^{k}(U):=\{S^{k}W\to \cdots \to \wedge^{k-1}(V)\otimes W\to \wedge^k(V)\}
  \end{align*}
  and $S^{k}(U)=\wedge^{k}(U)=0$ when $k<0$.
\end{definition}
\begin{definition}
  \label{projectivization}
  Let
  $$U:=\{W\xrightarrow{s} V\}$$
  be a two term complex of locally free sheaves over a scheme $X$ such that $W$ has rank $w$ and $V$ has rank $v$. Let $Z\subset \mathbb{P}_X(V)$ be the closed subscheme such that $\mathcal{O}_Z$ is the cokernel of  the composition of morphisms $$\rho^*W\otimes \mathcal{O}_{\mathbb{P}_X(V)}(-1)\xrightarrow{\rho^*(s)}\rho^*V\otimes \mathcal{O}_{\mathbb{P}_X(V)}(-1)\xrightarrow{taut} \mathcal{O}_{\mathbb{P}_X(V)}$$
where $\rho:\mathbb{P}_X(V)\to X$ is the projection morphism. We define $Z$ to be the projectivization of $U$ over $X$, denoted by
$$Z=\mathbb{P}_X(U)$$
if  $\mathcal{O}_Z$ is resolved by the Koszul complex:
$$0 \to \wedge^w\rho^*(W)\otimes \mathcal{O}_{\mathbb{P}_X(V)}(-w) \to \cdots \to \rho^*W\otimes \mathcal{O}_{\mathbb{P}_X(V)}(-1)\to \mathcal{O}_{\mathbb{P}_X(V)}\to \mathcal{O}_Z\to 0.$$
\end{definition}
When $Z$ is a projectivization of $U$ over $X$, we have a categorical projection lemma for  $R\rho_{*}(\mathcal{O}_Z(k))$:
\begin{lemma}[Categorical Projection Lemma]
  \label{prop:p}
 If $Z$ is the projectivization of $U$ over $X$ in \cref{projectivization}, then the tensor contraction
\begin{equation}
  \label{spec}
det(U)^{-1}\wedge^{w-v-k}W^\vee=\wedge^v(V^\vee)\wedge^{k+v}(W)\to \wedge^k(W)
\end{equation}
induces a morphism of complexes 
\begin{equation}
\label{cone}
  det(U)^{-1}\wedge^{w-v-k}(U^\vee)[-k] \to S^k(U)
\end{equation}
and $R\rho_*(\mathcal{O}_Z(k))$ is quasi-isomorphic to its cone.
\end{lemma}
\begin{proof}
   $\mathcal{O}_Z(k)$ is quasi-isomorphic to the complex
   $$\{\cdots \to \wedge^j\rho^*W\otimes \mathcal{O}_{\mathbb{P}_X(V)}(-j+k) \to \cdots \to \rho^*W\otimes \mathcal{O}_{\mathbb{P}_X(V)}(k-1)\to \mathcal{O}_{\mathbb{P}_X(V)}(k)\to 0\}.$$
   Consider the following two complexes
   \begin{align*}
     F_0=\{\cdots\to  \wedge^{k+v+1}\rho^*W\otimes \mathcal{O}_{\mathbb{P}_X(V)}(-v-1)\to \wedge^{k+v}\rho^*W\otimes\mathcal{O}_{\mathbb{P}_X(V)}(-v)\}[-v-k+1] \\
     F_1=\{\wedge^{k+v-1}\rho^*W\otimes\mathcal{O}_{\mathbb{P}_X(V)}(-v+1)\cdots\to  \rho^*W\otimes \mathcal{O}_{\mathbb{P}_X(V)}(k-1)\to \mathcal{O}_{\mathbb{P}_X(V)}(k)\}.
   \end{align*}
   Then the morphism $\wedge^{k+v}\rho^*W\otimes\mathcal{O}_{\mathbb{P}_X(V)}(-v)\to \wedge^{k+v-1}\rho^*W\otimes\mathcal{O}_{\mathbb{P}_X(V)}(-v+1)$ induces a morphism of $F_0\to F_1$ with the cone quasi-isomorphic to   $\mathcal{O}_{Z}(k)$.

 By Exercise \RN{3}.8.4 of Hartshorne \cite{hartshorne2013algebraic}, 
  \begin{equation}
    \label{1eq:2.4}
    R\rho_*(\mathcal{O}(j))=
    \begin{cases}
      S^{j}(V) & j\geq 0 \\
      0 & -v<j<0 \\
      detV^{-1}\otimes S^{-j-v}V^\vee[v-1] & j\leq -v.
    \end{cases}
  \end{equation}
  and thus
  $$R\rho_*F_0\cong det(U)^{-1}\wedge^{w-v-k}(U^\vee)[-k], \quad R\rho_*F_1\cong S^k(U).$$
  Hence $R\rho_*(\mathcal{O}_Z(k))$ is quasi-isomorphic to the cone
  $$det(U)^{-1}\wedge^{w-v-k}(U^\vee)[-k] \to S^k(U).$$
\end{proof}

\subsection{Nested Hilbert Schemes as Projectivization}
Recall the short exact sequence \cref{eq:resolve}:
$$0\to W_n\xrightarrow{s_n}V_n\to \mathcal{I}_n\to 0.$$
Nested Hilbert schemes can be realized as projectivizations, as in the following Propositions:

\begin{proposition}[Proposition 2.2 of \cite{ES98}]
  \label{1pr:1}
    The nested Hilbert scheme $S^{[n,n+1]}$ is the blow up of $\mathcal{Z}_n$ in $S^{[n]}\times S$:
     $$S^{[n,n+1]}\cong\mathbb{P}_{S^{[n]}\times S}(\mathcal{I}_n)$$ and is smooth of dimension $2n+2$. Moreover, $S^{[n,n+1]}$ is the projectivization of
     $$W_n\xrightarrow{s} V_n$$
     over $S^{[n]}\times S$. The tautological line bundle $\mathcal{L}$ is the restriction of $\mathcal{O}_{\mathbb{P}_{S^{[n]}\times S}(V_{n})}(1)$ to $S^{[n,n+1]}$.
\end{proposition}

\begin{corollary}
    \label{eq:Zn}
  $\mathcal{L}$ is the exceptional divisor of $S^{[n,n+1]}$ as the blow up of $\mathcal{Z}_n$, i.e. we have the short exact sequence:
  \begin{equation}
  0\to \mathcal{L}\to \mathcal{O}_{S^{[n,n+1]}} \to p_n^{-1}\mathcal{O}_{\mathcal{Z}_n}\to 0.  
\end{equation}
\end{corollary}
\begin{proof}
  It is obvious from the Proposition 7.13 of \cite{hartshorne2013algebraic}.
\end{proof}

Let $\overline{V_{n}}$ be the kernel of the surjective morphism
$$p_{n}^*V_{n}\to \mathcal{O}_{S^{[n,n+1]}}(1)=\mathcal{L}.$$
Then $\overline{V_{n}}$ is also locally free. The morphism $p_n^*(W_{n})\to p_n^*(V_{n})$ factors through $\overline{V_{n}}$ and induces a morphsim
$$\overline{V_{n}}^\vee\otimes \omega_S\to p_n^*(W_{n}^\vee)\otimes \omega_S.$$
\begin{proposition}[Proposition 4.15 of \cite{neguct2017w}]
 
  \label{1pr:2.4}
    The scheme $S^{[n-1,n,n+1]}$ is smooth of dimension $2n+1$. Moreover, it is the projectivization of
    $$\overline{V_{n}}^\vee\otimes \omega_S\to p_n^*(W_{n}^\vee)\otimes \omega_S$$
    over $S^{[n,n+1]}$.
    For the two tautological line bundles $\mathcal{L}_1$, $\mathcal{L}_2$, $\mathcal{L}_1=q_n^*(\mathcal{O}_{S^{[n,n+1]}}(1))$ and $\mathcal{L}_2$ is the restriction of $\mathcal{O}_{\mathbb{P}_{S^{[n,n+1]}}(p_n^*(W_{n}^\vee)\otimes \omega_S)}(-1)$ in $S^{[n,n+1]}$.

\end{proposition}
\subsection{The derived push-forward of line bundles on $S^{[n-1,n,n+1]}$}
\label{sec:derived}
 Given complexes $\{C_i| i\in \mathbb{Z}\}$ with morphisms $d_i:C_i\to C_{i+1}$ such that $d_i\circ d_{i+1}=0$, we will write 
$$\{\cdots \to C_{i+1}\to C_i\to \cdots \}$$
for the total complex of the double complex $C_{\bullet}$. In this subsection, we still abuse the notation to denote
$$\mathcal{I}_n:=\{W_{n}\to V_{n}\}, \qquad \mathcal{I}_n^\vee:=\{V_{n}^\vee\to W_{n}^\vee\}.$$

  \begin{lemma}
    \label{L1}
    We have the following formula for the derived push-forward $Rp_{n*}\mathcal{L}^{j}$:
    \begin{equation}
      \label{1eq:2.9}
      Rp_{n*}(\mathcal{L}^j)=
      \begin{cases}
        S^j(\mathcal{I}_n) & j \geq 0 \\
        \wedge^{-j-1}(\mathcal{I}_n^\vee)[-j-1] & j<0.
      \end{cases}
    \end{equation}
    \end{lemma}
    \begin{proof}
      By \cref{1pr:1}, $S^{[n,n+1]}$ is the projectivization of
      $$W_n\xrightarrow{s_n} V_n.$$
      Thus \cref{1eq:2.9} follows from \cref{prop:p}.
    \end{proof}

    \begin{lemma}
      \label{lem:4.5}
      We have the following formula for $Rq_{n*}\mathcal{L}_2^k$:
      \begin{equation}
        \label{eq:l2}
        Rq_{n*}(\mathcal{L}_2^k)=
        \begin{cases}
          \{\mathcal{L} \to \mathcal{O}_{S^{[n,n+1]}}\} & k=0 \\
          \omega_S^{-k}\otimes \{Lp_n^*(\wedge^{k}\mathcal{I}_n)\mathcal{L}\to \cdots \to Lp_n^*(\mathcal{I}_n)\otimes \mathcal{L}^{k}\to \mathcal{L}^{k+1}\}[2k-1] & k>0 \\
          \omega_S^{-k}\otimes \{\cdots \to \mathcal{L}^{-1}Lp_n^*(S^{-k-1}\mathcal{I}_n^\vee)[-1]\to Lp_n^*(S^{-k}\mathcal{I}_n^\vee) \} & k<0
        \end{cases}
      \end{equation}
      where $Lp_n^*:D^b(S^{[n]}\times S)\to D^b(S^{[n,n+1]})$ is the derived pull back morphism.
    \end{lemma}
    \begin{proof}
      By \cref{1pr:2.4}, $S^{[n-1,n,n+1]}$ is the projectivization of
    $$\overline{V_{n}}^\vee\otimes \omega_S\to p_n^*(W_{n}^\vee)\otimes \omega_S$$
    over $S^{[n,n+1]}$.
      We notice that $p_n^*(W_{n}^\vee)\otimes \omega_S$ and $\overline{V_{n}}^\vee\otimes\omega_S$ have the same rank and
\begin{align*}
  \frac{det(p_n^*(W_{n})\otimes \omega_S)}{det(\overline{V_{n}})\otimes \omega_S}= \frac{det(p_n^*(W_{n}))}{det(\overline{V_{n}})} = \frac{det(p_n^*(V_{n}))}{det(p_n^*(V_{n}))\mathcal{L}^{-1}}=\mathcal{L}.
\end{align*}
By \cref{prop:p}, we have
\begin{itemize}
\item When $k=0$,
  \begin{equation}
    \label{4eq:k=0}
Rq_{n*}(\mathcal{O}_{S^{[n-1,n,n+1]}})=\{\mathcal{L}\to \mathcal{O}_{S^{[n,n+1]}}\},    
\end{equation}
\item When $k>0$,  we have
  $$Rq_{n*}\mathcal{L}_2^k=\mathcal{L}\otimes \omega_S^{-k}\otimes \{S^{k}(p_{n}^*W_{n})\to \overline{V_{n}}\otimes S^{k-1}(p_n^*W_{n})\cdots \to \wedge^{k}(\overline{V_{n}})\}[k-1].$$
  By the resolution of $\wedge^k(\overline{V_{n}})$:
  \begin{equation*}
0\to \wedge^{k}(\overline{V_{n}})\to \wedge^k(V_{n})\to \wedge^{k-1}(V_{n})\otimes \mathcal{L}\to \cdots \to \mathcal{L}^k\to 0,
\end{equation*}
we have
\begin{equation*}
  Rq_{n*}(\mathcal{L}_2^k)=\omega_S^{-k}\otimes \{Lp_n^*(\wedge^{k}\mathcal{I}_n)\mathcal{L}\to \cdots \to Lp_n^*(\mathcal{I}_n)\otimes \mathcal{L}^{k}\to \mathcal{L}^{k+1}\}[2k-1].
\end{equation*}
\item When $k<0$,  we have
\begin{align*}
  Rq_{n*}(\mathcal{L}_2^k)=\omega_S^{-k}\otimes \{\cdots \to \overline{V_{n}}^\vee\otimes S^{-k-1}(p_n^*W_{n}^\vee)\to S^{-k}(p_n^*W_{n}^\vee)\}
\end{align*}
By the resolution of $\wedge^k(\overline{V_n}^\vee)$:
$$0\to \mathcal{L}^{k}\to \cdots \to \wedge^{-k-1}(V_{n}^\vee)\mathcal{L}^{-1}\to \wedge^k(V_{n}^\vee)\to \wedge^k(\overline{V_{n}}^\vee)\to 0,$$
we have
\begin{align}
  \label{1eq:4.14}
  Rq_{n*}(\mathcal{L}_2^k)&=\omega_S^{-k}\otimes \{\cdots \to \mathcal{L}^{-1}Lp_n^*(S^{-k-1}\mathcal{I}_n^\vee)[-1]\to Lp_n^*(S^{-k}\mathcal{I}_n^\vee) \}
\end{align}
\end{itemize}
\end{proof}
By \cref{L1} and \cref{lem:4.5}, we have
\begin{corollary}
  \label{cor}
  We have the following formula for $R(p_n\circ q_n)_*\mathcal{L}_1^{m-1-k}\mathcal{L}_2^k$:
  \begin{multline}
    R(p_n\circ q_n)_*(\mathcal{L}_1^{m-1-k}\mathcal{L}_2^k)= \\
      \begin{cases}
    \omega_S^{-k}\otimes \{\wedge^{-m}(\mathcal{I}_n^\vee)\to \cdots \to \wedge^{-m+k}(\mathcal{I}_n^\vee)S^{-k}(\mathcal{I}_n^\vee)\}[-m+k] & m\leq k \leq -1 \\
    \omega_S^{-k}\otimes \{\wedge^{k}(\mathcal{I}_n)S^{m-k}(\mathcal{I}_n)\to \cdots \to \mathcal{I}_n\otimes S^{m-1}(\mathcal{I}_n)\to S^{m}(\mathcal{I}_n)\}[2k-1] & 1\leq k\leq m.\\
    \end{cases}
  \end{multline}
\end{corollary}
  \begin{remark}
    When $m=k<0$, same as the computation in the Grothendieck group,  the complex $$\{\wedge^{-k}(\mathcal{I}_n^\vee)\to \cdots \to S^{-k}(\mathcal{I}_n^\vee)\}$$
    is quasi-isomorphic to $0$ and thus we have
      \begin{equation}
    \label{1:rmk}
R(p_n\circ q_n)_*\mathcal{L}^{-1}_1\mathcal{L}_2^{k}=0.
  \end{equation}
  \end{remark}

  \begin{proof}[Proof of \cref{thm:4.9}]
    It follows from \cref{4eq:lem2.11}, \cref{L1} and \cref{cor}.
  \end{proof}

\section{Singularity of the Minimal Model Program}
\label{sec:4}
In this section, we will review the singularities in the minimal model program from \cite{Kollar-Mori,MR3057950}. We use the notation $D$ to replace $\Delta$ in  \cite{Kollar-Mori,MR3057950}, as $\Delta$ is already used to denote the diagonal embedding in our paper.  For a normal variety, we will denote $K_X$ the canonical Weil divisor. We will denote by $\omega_X$ the dualizing sheaf when $X$ is Cohen-Macaulay. They coincide when $X$ is Gorenstein.

\subsection{Discrepancy and Classifacation of Singularities}

\begin{definition}[Definition 2.25 of \cite{Kollar-Mori} or Definition 2.4 of \cite{MR3057950}, Discrepancy]
  Let $(X,D)$ be a pair where $X$ is a normal variety and $D=\sum a_iD_i,a_i\in \mathbb{Q}$ is a sum of distinct prime divisors. Assume that $m(K_X+D)$ is Cartier for some $m>0$. Suppose $f:Y\to X$ is a birational morphism from a normal variety $Y$. Let $E\subset Y$ denote the exceptional locus of $f$ and $E_i\subset E$ the irreducible exceptional divisors. The two line bundles
  $$\mathcal{O}(m(K_Y+f^{-1}D))|_{Y-E} \text{ and } f^*\mathcal{O}_{X}(m(K_X+D))|_{Y-E}$$
  are naturally isomorphic. Thus there are rational numbers $a(E_i,X,D)$ such that $ma(E_i,X,D)$ are integers, and
  $$\mathcal{O}_{Y}(m(K_Y+f^{-1}D))\cong f^*\mathcal{O}_X(m(K_X+D))\otimes \mathcal{O}_Y(\sum_{i} ma(E_i,X,D)E_i).$$
  $a(E,X,D)$ is called the discrepancy of $E$ with respect to $(X,D)$. We define the center of $E$ in $X$ by
$$center_X(E):=f(E).$$
  When $D=0$, then $a(E_i,X,D)$ depends only on $E_i$ but not on $f$.
\end{definition}
\begin{lemma}[Lemma 2.29 of \cite{Kollar-Mori}]
  \label{blowup}
  Let $X$ be a smooth variety and $D=\sum a_i D_i$ a sum of distinct prime divisors. Let $Z\subset X$ be a closed subvariety of codimension $k$. Let $p:Bl_ZX\to X$ be the blow up of $Z$ and $E\subset Bl_ZX$ the irreducible component of the exceptional divisor which dominates $Z$, (if $Z$ is smooth, then this is the only component). Then
  $$a(E,X,D)=k-1-\sum_i a_i mult_ZD_i$$
  where $mult_ZD_i$ is the multiplicity of $D_i$ in $Z$.
\end{lemma}

\begin{definition}[Definition 2.34 and 2.37 of \cite{Kollar-Mori}, or Definition 2.8 of \cite{MR3057950}]
  Let $(X,D)$ be a pair where $X$ is a normal variety and $D=\sum a_iD_i$ is a sum of distinct prime divisors where $a_i\in \mathbb{Q}$ and $a_i\leq 1$. Assume that $m(K_X+D)$ is Cartier for some $m>0$. We say that $(X,D)$ is

  \begin{minipage}{0.3 \linewidth} 
    \begin{equation*}
      \begin{cases}
     terminal \\
     canonical \\
     klt \\
     plt \\
     dlt \\
     lc \\
    \end{cases}
    \end{equation*}
  \end{minipage}
    \begin{minipage}{0\linewidth}
    \begin{equation*}
      \text{if } a(E,X,D) \text{ is }
      \begin{cases}
        >0, \text{ for every exceptional } E\\
        \geq 0,  \text{ for every exceptional } E\\
        >-1,  \text{ for every } E \\
        >-1,  \text{ for every exceptional } E \\
        >-1,  \text{ if } center_XE\subset \text{non-snc}(X, D)\\
        \geq-1. \text{ for every } E\\
      \end{cases}
    \end{equation*}
  \end{minipage}
  
  Here klt is short for ``Kawamata log terminal'', plt for ``pure log terminal'' and lc for ``log canonical''.Above, non-snc$(X,D)$ denotes the set of points where $(X,D)$ is not simple normal crossing(snc for short).

  We say that $X$ is terminal (canonical,etc) if and only if $(X,0)$ is terminal (canonical,etc).

  Each class contains the previous one, except canonical does not imply klt if $D$ contains a divisor with coefficient $1$.
\end{definition}
\begin{theorem}[Theorem 5.50 of \cite{Kollar-Mori}, or Theorem 4.9 of \cite{MR3057950}, Inversion of adjunction]
  \label{inversion}
    Let $X$ be normal and $S\subset X$ a normal Weil divisor which is Cartier in codimension $2$. Let $B$ be an effective $\mathbb{Q}$-divisor and assume that $K_X+S+B$ is $\mathbb{Q}$-Cartier. Then $(X,S+B)$ is plt near $S$ iff $(S,B|_S)$ is klt.
  \end{theorem}

\subsection{Rational Singularities}

\begin{definition}[Definition 5.8 of \cite{Kollar-Mori}]
  \label{rational}
    Let $X$ be a variety of a field of characteristic $0$. We say that $X$ is a rational singularity if there exists one resolution of singularities $f:Y\to X$, 
    \begin{enumerate}
    \item $f_*\mathcal{O}_Y=\mathcal{O}_X$(equivalently, $X$ is normal), and 
    \item $R^if_*\mathcal{O}_Y=0$ for $i>0$.
    \end{enumerate}
  \end{definition}
  \begin{remark}
    By Theorem 5.10 of \cite{Kollar-Mori}, if $X$ is a rational singularity, then for all resolution of singularities $f:Y\to X$, 
    \begin{enumerate}
    \item $f_*\mathcal{O}_Y=\mathcal{O}_X$,
    \item $R^if_*\mathcal{O}_Y=0$ for $i>0$.
    \end{enumerate}
  \end{remark}
  \begin{theorem}[Theorem 5.22 of \cite{Kollar-Mori}]
     \label{thm:5.22}
    Let $X$ be a normal variety over a field of characteristic $0$. If $X$ is a canonical singularity, then $X$ is a rational singularity. If $X$ is Gorenstein, then $X$ is a canonical singularity if $X$ is a rational singularity.
  \end{theorem}

  Here are also some examples of rational singularities:

 \begin{example}
    By \cite{MR3484148}, the universal closed subscheme $\mathcal{Z}_{n}$ is a rational singularity.
  \end{example}
  \begin{example}
    \label{canonical}
    The morphism $q_n:S^{[n-1,n,n+1]}\to S^{[n,n+1]}$ factors through $p_n^{-1}\mathcal{Z}_n$ and induces a resolution of $p_n^{-1}\mathcal{Z}_n$. By \cref{lem:4.5} and \cref{eq:Zn}, $p_n^{-1}\mathcal{Z}_n$ is a Gorenstein rational singularity and thus is a canonical singularity by \cref{thm:5.22}.
  \end{example}
\subsection{Semi-dlt Pairs}

\begin{definition}[Definition 1.10 of \cite{MR3057950}, Semi-snc pairs]
Let $W$ be a regular subscheme and $\sum_{i\in I}E_i$ a snc divisor on $W$. Write $I=I_Y\cup I_D$ as a disjoint union. Set $Y:=\sum_{i\in I_V}E_i$ as a subscheme and $D_Y:=\sum_{i\in I_D}a_iE_i|_Y$ as a divisor on $Y$ for some $a_i\in \mathbb{Q}$. We call $(Y,D_Y)$ an embedded semi-snc pair. A pair $(X,D)$ is called semi-snc if it is Zariski locally isomorphic to an embedded semi-snc pair.
\end{definition}

\begin{example}
  We have the following three examples of semi-snc pairs $(X,D)$:
  \begin{enumerate}
  \item $X=\{z=0\} \subset \mathbb{A}^3$ and $D=a_x(x|_X=0)+a_y(y|_X=0)$.
  \item $X=\{yz=0\}\subset \mathbb{A}^3$ and $D=a_x(x|_X=0)$.
  \item $X=\{xyz=0\}\subset \mathbb{A}^3$ and $D=0$.
  \end{enumerate}
\end{example}
\begin{definition}[Definition 5.1 of \cite{MR3057950}, Demi-normal schemes]
  A scheme $X$ is called demi-normal if it is $S_2$ and codimension $1$ points are either regular points or nodes. Here we say a scheme $X$ has a node at a point $x\in X$ if its local ring $\mathcal{O}_{x,X}$ can be written as $R/(f)$ where $(R,m)$ is a regular local ring of dimension $2$, $f\in m^2$ and $f$ is not a square in $m^2/m^3$.
\end{definition}

\begin{definition}[Section 5.2 of \cite{MR3057950}, Conductor]
  \label{conductor}
  Let $X$ be a reduced scheme and $\pi:\bar{X}\to X$ its normalization. The conductor ideal
  $$cond_X:=\mathcal{H}om(\pi_*\mathcal{O}_X,\mathcal{O}_X)\subset \mathcal{O}_X$$
  is the largest ideal sheaf on $X$ that is also an ideal sheaf on $\bar{X}$. We write it as $cond_{\bar{X}}$ when we view the conductor as an ideal sheaf on $\bar{X}$. The conductor subschemes are defined as
  $$T:=Spec_X(\mathcal{O}_X/cond_X) \qquad \bar{T}:=Spec_{\bar{X}}(\mathcal{O}_{\bar{X}}/cond_{\bar{X}}).$$
\end{definition}

\begin{definition}[Definition-Lemma 5.10 of \cite{MR3057950}]
  Let $X$ be a demi-normal scheme with normalization $\pi:\bar{X}\to X$ and conductors $T\subset X$ and $\bar{T}\subset \bar{X}$. Let $D$ be an effective $\mathbb{Q}$-divisor whose support does not contain any irreducible components of $T$ and $\bar{D}$ the divisorial part of $\pi^{-1}(D)$. The pair $(X,D)$ is called semi log canonical or slc if
  \begin{enumerate}
  \item $K_X+D$ is $\mathbb{Q}$-Cartier, and 
  \item $(\bar{X},\bar{D}+\bar{T})$ is lc.
  \end{enumerate}
\end{definition}

\begin{definition}[Definition 5.19 of \cite{MR3057950}]
  \label{semi-dlt}
  An slc pair $(X,D)$ is semi-divisorial log terminal or semi-dlt if $a(E,X,D)>-1$ for every exceptional divisor $E$ over $X$ such that $(X,D)$ is not snc at the generic point of $center_XE$.
\end{definition}

\begin{example}
  \label{snc}
  A semi-snc pair $(X,D)$ is always semi-dlt.
\end{example}

\begin{proposition}[Proposition 5.20 of \cite{MR3057950}]
  \label{sdlt}
  Let $(X,D)$ be a demi-normal pair over a field of characteristic $0$. Assume that the normalization $(\bar{X},\bar{T}+\bar{D})$ is dlt and there is a codimension $3$ set $W\subset X$ such that $(X\backslash W,D|_{X\backslash W})$ is semi-dlt. Then
  \begin{enumerate}
  \item the irreducible components of $X$ are normal,
  \item $K_X+D$ is $\mathbb{Q}$-Cartier and 
  \item $(X,D)$ is semi-dlt.
  \end{enumerate}

\end{proposition}

\section{The Quadrulple/Triple Moduli Spaces and the Minimal Model Program} 
\label{sec:5}
In this section we introduce the triple moduli spaces $\mathfrak{Z}_+,\mathfrak{Z}_{-}$ and the quadruple moduli space $\mathfrak{Y}$ which parameterize diagrams:

\begin{minipage}{0.5\linewidth}
  \begin{equation}
  \begin{tikzcd}
    &\mathcal{I}_n \\
    \mathcal{I}_{n+1} \ar[ru,hook,"x"] \ar[rd,hook,"y",swap] & \\
    & \mathcal{I}_n'
  \end{tikzcd}
\end{equation}
\end{minipage}
\begin{minipage}{0.5\linewidth}
  \begin{equation}
  \begin{tikzcd}
    \mathcal{I}_{n} \ar[rd,hook,"y"] & \\
    & \mathcal{I}_{n-1} \\
    \mathcal{I}_n' \ar[ru,hook,"x",swap] &
  \end{tikzcd}
\end{equation}
\end{minipage}

\begin{equation}
  \begin{tikzcd}
    &\mathcal{I}_n  \ar[rd,hook,"y"] & \\
    \mathcal{I}_{n+1} \ar[ru,hook,"x"] \ar[rd,hook,swap,"y"] & &   \mathcal{I}_{n-1}\\
    & \mathcal{I}_n' \ar[ru,hook,"x",swap] & 
  \end{tikzcd}
\end{equation}
respectively, of ideal sheaves where each successive inclusion is colength $1$ and supported at the point indicated on the diagrams. We consider line bundles $\mathcal{L}_{1},\mathcal{L}_{2},\mathcal{L}_{1}',\mathcal{L}_2'$ over triple/quadruple moduli spaces with fiber $\mathcal{I}_{n+1}/\mathcal{I}_{n},\mathcal{I}_n/\mathcal{I}_{n-1},\mathcal{I}_{n+1}/\mathcal{I}_{n}',\mathcal{I}_{n}'/\mathcal{I}_{n-1}$ respectively.

We consider the Cartesian diagram:
\begin{equation}
  \label{2eq:3.4}
  \begin{minipage}{0.4\linewidth}
  \begin{tikzcd}
    \mathfrak{Y} \ar{r}{\alpha_+} \ar{d}{\alpha_-} \ar{rd}{\theta}& \mathfrak{Z}_+ \ar{d}{\beta_+} \\
    \mathfrak{Z}_- \ar{r}{\beta_-} & S^{[n]}\times S^{[n]}\times S \times S
  \end{tikzcd}
\end{minipage}
\begin{minipage}{0.5\linewidth}
    \begin{tikzcd}
      (\mathcal{I}_{n-1},\mathcal{I}_n,\mathcal{I}_n',\mathcal{I}_{n+1},x,y) \ar{r}{\alpha_+} \ar{d}{\alpha_-} \ar{rd}{\theta}& (\mathcal{I}_n,\mathcal{I}_n',\mathcal{I}_{n+1},x,y)  \ar{d}{\beta_+} \\
      (\mathcal{I}_{n-1},\mathcal{I}_n,\mathcal{I}_n',x,y) \ar{r}{\beta_-} & (\mathcal{I}_n,\mathcal{I}_n',x,y).
    \end{tikzcd}
  \end{minipage}
\end{equation}
\cref{diagram4.1} is the restriction of \cref{2eq:3.4} to the diagonal $\Delta:S^{[n]}\times S\to S^{[n]}\times S^{[n]}\times S\times S$.

\begin{example}
  \label{example6.101}
    When $n=1$, then $\mathfrak{Y}=Bl_{\Delta_S}(S\times S)$,  where $\Delta_S:S\to S\times S$ is the diagonal embedding.  $\mathfrak{Z}_{-}=S\times S$ and $\mathfrak{Z}_+$ is induced by the Cartesian diagram
  \begin{equation}
    \begin{tikzcd}
      \mathfrak{Z}_{+}\ar{r}\ar{d} & Bl_{\Delta_S}(S\times S) \ar{d} \\
      Bl_{\Delta_S}(S\times S) \ar{r} & Bl_{\Delta_S}(S\times S)/\mathbb{Z}_{2}.
    \end{tikzcd}
  \end{equation}
  $\mathfrak{Z}_+$ has two irreducible components and each one is isomorphic to $Bl_{\Delta}(S\times S)$.
\end{example}

The main purpose of this section is to compute $R\alpha_{+*}\mathcal{O}_{\mathfrak{Y}}$ and $R\alpha_{-*}\mathcal{O}_{\mathfrak{Y}}$ explicitly.
\begin{proposition}
  \label{prop5.2}
  We have the formula
  \begin{equation}
R\alpha_{-*}\mathcal{O}_{\mathfrak{Y}}=\mathcal{O}_{\mathfrak{Z}_-}, \quad R\alpha_{+*}\mathcal{O}_{\mathfrak{Y}}=\mathcal{O}_{W_0}
  \end{equation}
  where $W_0$ will be defined in \cref{sec:5.3}.
\end{proposition}
  \cref{prop5.2} follows from \cref{5cor1} and \cref{5cor2}, which will be proved later in this section.

\subsection{The Geometry of $\mathfrak{Y}$}

  \begin{theorem}[Proposition 2.28 and Proposition 5.28 of \cite{neguct2018hecke}]
  \label{2pr:3.1}
  The scheme $\mathfrak{Y}$ is smooth of dimension $2n+2$. The closed embedding:
  $$\Delta_{\mathfrak{Y}}:S^{[n-1,n,n+1]}\to \mathfrak{Y} \quad (\mathcal{I}_{n-1},\mathcal{I}_n,\mathcal{I}_{n+1},x)\to (\mathcal{I}_{n-1},\mathcal{I}_n,\mathcal{I}_n,\mathcal{I}_{n+1},x,x)$$
 is a regular closed subscheme of codimension $1$.  If we abuse the notation to denote  $S^{[n-1,n,n+1]}$ by $\Delta_{\mathfrak{Y}}$, then  the morphism of coherent sheaves over $\mathfrak{Y}\times S$:
$$\mathcal{I}_n/\mathcal{I}_{n+1}\to \mathcal{I}_{n-1}/\mathcal{I}_{n+1}\to \mathcal{I}_{n-1}/\mathcal{I}_{n}', \quad \mathcal{I}_n'/\mathcal{I}_{n+1}\to \mathcal{I}_{n-1}/\mathcal{I}_{n+1}\to \mathcal{I}_{n-1}/\mathcal{I}_{n}$$
induce short exact sequences
\begin{align*}
  0\to \mathcal{L}_1\to \mathcal{L}_2' \to \mathcal{L}_2'\mathcal{O}_{\Delta_\mathfrak{Y}}\to 0 \\
  0\to \mathcal{L}_1'\to \mathcal{L}_2 \to \mathcal{L}_2\mathcal{O}_{\Delta_\mathfrak{Y}}\to 0.
\end{align*}
Moreover, $\mathcal{L}_1\mathcal{L}_2'^{-1}=\mathcal{L}_1'\mathcal{L}_2^{-1}=\mathcal{O}(-\Delta_{\mathfrak{Y}})$. The normal bundle $N_{\mathfrak{Y}/\Delta_{\mathfrak{Y}}}=\mathcal{L}_1^{-1}\mathcal{L}_2$.
\end{theorem}
\begin{remark}
  From now on we will abuse the notation to denote $S^{[n-1,n,n+1]}$ by $\Delta_{\mathfrak{Y}}$.
\end{remark}

\begin{lemma}[Claim 3.8 of \cite{negut2017shuffle}]
  \label{claim3.8}
  Let $U$ be the complement of $\Delta$. Then $\alpha_+$ and $\alpha_-$ in  \cref{2eq:3.4} are isomorphisms when restricting to $U$.
\end{lemma}
\subsection{The Geometry of $\mathfrak{Z}_-$}

\begin{proposition}
  \label{1pr:3.1}
  The scheme $\mathfrak{Z}_-$ is an irreducible $2n+2$ dimensional locally complete intersection scheme and is a canonical singularity.
\end{proposition}
\begin{proof}
  By \cref{claim3.8}, $\alpha^{-1}_-(U)\cong \theta^{-1}(U)$ is a $2n+2$ dimensional smooth open subscheme and the complement is the $2n$ dimensional closed subscheme $S^{[n-1,n]}$ by \cref{2eq:3.4}. Hence $\mathfrak{Z}_-$ is a $2n+2$ dimensional $R_1$ scheme in the sense of Serre's condition.

  On the other hand, $$\mathfrak{Z}_{-}=S^{[n-1,n]}\times_{S^{[n-1]}}S^{[n-1,n]}$$
  has the expected dimension also equal to $2n+2$.  Thus $\mathfrak{Z}_-$ is a locally complete intersection scheme and thus is Cohen-Macaulay and normal.

  Now we prove that $\mathfrak{Z}_{-}$ is a canonical singularity. The complement of $\theta^{-1}(U)$ in $\mathfrak{Y}$ is $\Delta_{\mathfrak{Y}}=S^{[n-1,n,n+1]}$ and thus there exists $a\in \mathbb{Q}$ such that
  $$\alpha_-^*K_{\mathfrak{Z}_-}=K_{\mathfrak{Y}}+a\mathcal{O}(\Delta_{\mathfrak{Y}})$$

  We denote $\mathfrak{Y}_{1}$ the quadruple moduli space $\mathfrak{Y}$ when $n=1$. Let
  $$V_{1}:=\{(\mathcal{I}_{n-1},(\mathcal{I}_{1},\mathcal{I}_{1}',\mathcal{I}_2,x,y))\in S^{[n-1]}\times \mathfrak{Y}_{1}|(\mathcal{I}_{n-1},x)\notin \mathcal{Z}_{n-1} \text{ and } (\mathcal{I}_{n-1},y)\notin \mathcal{Z}_{n-1}\}$$
  where $\mathcal{Z}_{n-1}$ is the universal closed subscheme of $S^{[n-1]}\times S$. The morphism
  $$(\mathcal{I}_{n-1},(\mathcal{I}_{1},\mathcal{I}_{1}',\mathcal{I}_2,x,y))\to (\mathcal{I}_{n-1},\mathcal{I}_{n-1}\cap \mathcal{I}_1,\mathcal{I}_{n-1}\cap \mathcal{I}_{1}',\mathcal{I}_{n-1}\cap \mathcal{I}_{2},x,y)$$
  induces an open embedding $V_1\subset \mathfrak{Y}$.

  We define
  $$V_2:=\{(\mathcal{I}_{n-1},x,y)\in S^{[n]}\times S\times S|(\mathcal{I}_{n-1},x)\notin \mathcal{Z}_{n-1} \text{ and } (\mathcal{I}_{n-1},y)\notin \mathcal{Z}_{n-1}\}, $$
  and the morphism
  $$(\mathcal{I}_{n-1},x,y)\to (\mathcal{I}_{n-1},\mathcal{I}_{n-1}\cap \mathcal{I}_{x},\mathcal{I}_{n-1}\cap \mathcal{I}_{y})$$
  induces an open embedding $V_2\subset \mathfrak{Y}_{-}$. $\alpha_{-}^{-1}(V_2)=V_1$ and we have the Cartesian diagram:
  \begin{equation*}
    \begin{tikzcd}
      V_1 \ar{r}\ar{d} & \mathfrak{Y}_{1} \ar{d} \\
      V_2 \ar{r} & S\times S
    \end{tikzcd}
  \end{equation*}
  where all horizontal morphisms are natural projections. By \cref{example6.101}, $\mathfrak{Y}_{2}=Bl_{\Delta_{S}}(S\times S)$. Thus by \cref{blowup}, $a=1$ and $\mathfrak{Z}_-$ is a canonical singularity.
\end{proof}

\begin{corollary}
  \label{5cor1}
  We have the formula
  $$R\alpha_{-*}(\mathcal{O}_{\mathfrak{Y}})=\mathcal{O}_{\mathfrak{Z}_{-}}.$$
\end{corollary}

\subsection{The Geometry of $\mathfrak{Y}_+$}
\label{sec:5.3}
The geometry of $\mathfrak{Y}_{+}$ is more complicated, as it is no longer irreducible. We define $W_0$ as a closed subscheme of $\mathfrak{Z}_{+}$ by
$$W_0:=\{(\mathcal{I}_n,\mathcal{I}_{n}',\mathcal{I}_{n+1},x,y)\in \mathfrak{Z}_+|(\mathcal{I}_{n}',y)\in \mathcal{Z}_{n} \text{ and } (\mathcal{I}_{n},x)\in \mathcal{Z}_{n}\}.$$
$W_0$ is the closure of  $\beta_+^{-1}(U)$ in $\mathfrak{Z}_{+}$. We define $W_{1}:=S^{[n,n+1]}$.

\begin{proposition}
  \label{2pr:structure}
  The scheme $\mathfrak{Z}_+$ is a locally complete intersection scheme (and hence Cohen-Macaulay) of dimension $2n+2$, with two irreducible components $W_0$ and $W_1$. $W_0\cap W_1=p_n^{-1}\mathcal{Z}_n$. 
\end{proposition}
\begin{proof}
  $W_1$ is $2n+2$ dimensional and the complement of $W_1$ in $\mathfrak{Z}_+$ is $\beta_+^{-1}(U)=\theta^{-1}(U)$, which is also $2n+2$ dimensional by \cref{claim3.8}. Thus $\mathfrak{Z}_+$ is $2n+2$ dimensional and
  $$\mathfrak{Z}_{+}=S^{[n,n+1]}\times_{S^{[n+1]}}S^{[n,n+1]}$$
  has expected dimension also equal to $2n+2$. Hence $\mathfrak{Z}_+$ is a locally complete intersection scheme.

Any closed point of $W_{0}\cap W_1$ corresponds to $(\mathcal{I}_{n},\mathcal{I}_{n+1},x)\in S^{[n,n+1]}$ such that $x$ has length $\geq 1$ in $\mathcal{O}/\mathcal{I}_n$ and thus is in $p_n^{-1}\mathcal{Z}_n$.
\end{proof}
\begin{example}
  \label{example6.9}
  When $n=2$, by \cref{example:2.5}, $\mathcal{Z}_2=S^{[1,2]}$. Thus for any point
  $(\mathcal{I}_2',\mathcal{I}_2,\mathcal{I}_{3},x,y)\in W_0$, there exists a unique ideal sheaf $\mathcal{I}_1$ with two short exact sequences
  $$0\to \mathcal{I}_2'\to \mathcal{I}_1\to k_y\to 0 \quad 0\to \mathcal{I}_2\to \mathcal{I}_1\to k_x\to 0.$$
  Thus $W_0\cong \mathfrak{Y}$.

   $W_0\cap W_1=S^{[1,2,3]}$ by \cref{example:2.5} and is smooth. So $(\mathfrak{Z}_+,0)$ is a semi-snc pair.
\end{example}
 Consider the closed subscheme $W_2\subset p_{n}^{-1}\mathcal{Z}_n$
  $$W_2:=\{(\mathcal{I}_{n},\mathcal{I}_{n+1},x)\in W_1|\text{ the length of } k_x \text{ in } \mathcal{O}/\mathcal{I}_n\geq 3.\}.$$
\begin{lemma}
 
  The scheme $q_n^{-1}(W_2)$ and $W_2$ have dimension less or equal to $2n-1$.
\end{lemma}
\begin{proof}
  The scheme
  $$S^{[n-2,n-1,n,n+1]}:=S^{[n-2,n-1,n]}\times_{S^{[n-1,n]}}S^{[n-1,n,n+1]}$$
  has dimension $2n-1$ by (5.21) of \cite{neguct2018hecke}. The image of the projection morphism
  $$S^{[n-2,n-2,n,n+1]}\to S^{[n,n+1]}$$
  is $W_2$ and the image of the projecton morphism
  $$S^{[n-2,n-2,n,n+1]}\to S^{[n-1,n,n+1]}$$
  is $q_{n}^{-1}(W_2)$. Hence $q_n^{-1}(W_2)$ and $W_2$ have dimension less or equal to $2n-1$.
\end{proof}

As the morphism $\alpha_{+}:\mathfrak{Y}\to \mathfrak{Z}_{+}$ factors through $W_0$, we will abuse the notation to denote the morphism $\alpha_{+}:\mathfrak{Y}\to W_0$.
\begin{lemma}
  The morphism $\alpha_{+}:\mathfrak{Y}\to W_0$ is an isomorphism when restricting to $W_0-W_2$.
\end{lemma}
\begin{proof}
   We denote $\mathfrak{Y}_{2}$ the quadruple moduli space when $n=2$. Let
   $$V_3:=\{(\mathcal{I}_{n-2},(\mathcal{I}_1,\mathcal{I}_2,\mathcal{I}_{2}',\mathcal{I}_3,x,y))\in S^{[n-1]}\times \mathfrak{Y}_2|\mathcal{I}_{n-2}+\mathcal{I}_3=\mathcal{O}\}$$
   and $V_4:=\alpha_{+}V_3$. The morphism
  $$(\mathcal{I}_{n-2},(\mathcal{I}_1,\mathcal{I}_2,\mathcal{I}_{2}',\mathcal{I}_3,x,y))\to (\mathcal{I}_{n-2}\cap\mathcal{I}_1,\mathcal{I}_{n-2}\cap\mathcal{I}_2,\mathcal{I}_{n-2}\cap\mathcal{I}_{2}',\mathcal{I}_{n-2}\cap\mathcal{I}_3,x,y)$$
  induces an open emebedding $V_3\subset \mathfrak{Y}$. By \cref{example6.9},
  $$V_3\cong V_4$$
  By \cref{claim3.8}, $\alpha_{+}$ is also an isomorphism when restricting to $\beta_+^{-1}(U)$. $W_2$ is the complment of $\beta_+^{-1}(U)\cup  V_4$ in $W_0$.
\end{proof}

Let $\overline{W_0}$ be the normalization of $W_0$ and $\overline{W_0\cap W_1},\overline{W_2}$ the preimage of $W_0\cap W_1$ and $W_2$ in the normalization. The morphism $\alpha_{+}:\mathfrak{Y}\to W_0$ will factor through
$\overline{\alpha_{+}}:\mathfrak{Y}\to \overline{W_0}$ 
\begin{lemma}
  \label{lem:dlt}
  The scheme $\overline{W}_0$ is a canonical singularity and the pair $(\overline{W_0},\overline{W_0\cap W_1})$ is plt.
\end{lemma}
\begin{proof}
    The codimension of $\overline{W_2}$ in $\overline{W_0}$ is $3$ and $\overline{\alpha_{+}}$ is an isomorphism outside of $\overline{W_2}$. The preimage of $\overline{W_0\cap W_1}$ in $\mathfrak{Y}$ is $\Delta_{\mathfrak{Y}}$, which is a smooth divisor of $\mathfrak{Y}$.  Hence $\overline{W}_0$ is a canonical singularity and the pair $(\overline{W_0},\overline{W_0\cap W_1})$ is plt.
\end{proof}
\begin{proposition}
  \label{besemidlt}
  
  The pair $(\mathfrak{Z}_+,0)$ is semi-dlt and  $W_0$ is normal.
\end{proposition}
\begin{proof}
  Let
  $$V_5:=\mathcal{I}_{n-2},(\mathcal{I}_2,\mathcal{I}_{2}',\mathcal{I}_3,x,y))\in S^{[n-1]}\times \mathfrak{Z}_2^{+}|\mathcal{I}_{n-2}+\mathcal{I}_3=\mathcal{O}\}$$
  where $\mathfrak{Z}_2^{+}=S^{[2,3]}\times_{S^{[2]}}S^{[2,3]}$ is the triple moduli space when $n=2$. $V_5$ is an open subscheme of $\mathfrak{Z}_{+}$. The pair$(V_5,0)$ is a semi-snc by \cref{example6.9} and so is $(V_5\cup \beta_{+}^{-1}(U),0)$. $W_2$ is the complement of $V_5\cup \beta_{+}^{-1}(U)$ in $\mathfrak{Z}_{+}$ and 
  $$codim_{W_2}\mathfrak{Z}_+=3.$$

  By \cref{conductor}, $W_0\cap W_1$ is the conductor subscheme of $\mathfrak{Z}_{+}$. By \cref{canonical}, $W_0\cap W_1$ is a canonical singularity. By the inversion of adjunction theorem \cref{inversion}, the pair $(W_1,W_0\cap W_1)$ is plt and thus dlt.

 By \cref{lem:dlt}, $(\overline{W_0},\overline{W_0\cap W_1})$ is also a plt pair and thus a dlt pair. 

By \cref{sdlt}, $(\mathfrak{Z}_+,0)$ is a semi-dlt pair, and $W_0$ is normal. 
\end{proof}

\begin{corollary}
  \label{5cor2}
  We have the formula
  $$R\alpha_{+*}(\mathcal{O}_{\mathfrak{Y}})=\mathcal{O}_{W_0}.$$
\end{corollary}

\section{The Proof of \cref{thm:1.1}}
\label{sec:6}
In this section, we prove \cref{thm:1.1}.

\begin{definition}
  We will write $S_1,S_2$ for two copies of $S$, in order to emphasize  the  factors of $S\times S$ and write $\mathcal{M}_1$, $\mathcal{M}_2$ and $\mathcal{M}_3$ for three copies of $\mathcal{M}$, in order to emphasize the factors of $\mathcal{M}\times \mathcal{M}\times \mathcal{M}$. Given $P\in D^{b}(\mathcal{M}_{1}\times \mathcal{M}_{2}\times S_1)$ and $Q\in D^{b}(\mathcal{M}_{2}\times \mathcal{M}_{3}\times S_2)$, we define the composition $QP\in D^{b}(\mathcal{M}_{1}\times \mathcal{M}_{3}\times S_1\times S_2)$ by
$$QP:=R\pi_{13_{*}}(L\pi_{12}^*P\otimes L\pi_{23}^*Q),$$
where $\pi_{12},\pi_{23}$ and $\pi_{13}$ are the projections from $\mathcal{M}_1\times \mathcal{M}_2\times \mathcal{M}_3\times S_1\times S_2$ to $\mathcal{M}_1\times \mathcal{M}_2\times S_1$, $\mathcal{M}_2\times \mathcal{M}_3\times S_2$ and $\mathcal{M}_1\times \mathcal{M}_3\times S_1\times S_2$ respectively.
\end{definition}

\begin{theorem}
  \label{thm:7.1}
  Given two integers $m$ and $r$ and let
  \begin{align*}
    e_{m-r}=\mathcal{L}_1^{m-r}\mathcal{O}_{S^{[n,n+1]}}\in D^b(S^{[n]}\times S^{[n+1]}\times S) \\
    f_r=\mathcal{L}_1^{k-1}\mathcal{O}_{S^{[n,n+1]}}[1]\in D^b(S^{[n+1]}\times S^{[n]}\times S).
  \end{align*}

  Then
  \begin{enumerate}
  \item If $m>0$, then there are $b_{m,r}^k\in D^b(S^{[n]}\times S^{[n]}\times S\times S)$ for $0\leq k \leq m$ such that $b_{m,r}^{m}=e_{m-r}f_r$ and $b_{m,r}^{0}=f_re_{m-r}$, with exact triangles
    $$\mathfrak{B}_{m,r}^{k}:R\Delta_*(h_{m,k}^{+})[-1]\to b_{m,r}^{k}\to b_{m,r}^{k+1}\to R\Delta_*(h_{m,k}^{+})$$ 
  \item If $m<0$, then there are $c_{m,r}^r\in D^{b}(S^{[n]}\times S^{[n]}\times S \times S)$ for $m\leq k\leq 0$ such that $c_{m,r}^{m}=e_{m-r}f_r$ and  $c_{m,r}^{0}=f_re_{m-r}$, and exact triangles
    $$\mathfrak{C}_{m,r}^k:R\Delta_*(h_{m,k}^-)[-1]\to c_{m,r}^{k-1}\to c_{m,r}^{k}\to R\Delta_*(h_{m,k}^-)$$
        \item If $m=0$, there is an explicit isomorphism $f_re_{-r}=e_rf_{-r}\oplus \mathcal{O}_{\Delta}[1]$.
  \end{enumerate}
\end{theorem}
\cref{thm:1.1} follows from \cref{thm:7.1} and \cref{thm:4.9}. We will prove \cref{thm:7.1} later in this section.

\subsection{$e_{m-r}f_r$ and $f_re_{m-r}$ Revisited}
As $\mathfrak{Z}_{-}$ and $\mathfrak{Z}_{+}$ are both Cohen-Macaulay of expected dimension, we have the following formula:
  \begin{align*}
     e_{m-r}f_{r}=R\beta_{-*}(\mathcal{L}_2'^{m-r}\mathcal{L}_2^{r-1}\mathcal{O}_{\mathfrak{Z}_{-}})[1]\in D^b(S^{[n]}\times S^{[n]}\times S \times S) \\
    f_re_{m-r}=R\beta_{+*}(\mathcal{L}_1^{m-r}\mathcal{L}_1'^{r-1}\mathcal{O}_{\mathfrak{Z}_{+}})[1]\in D^b(S^{[n]}\times S^{[n]}\times S \times S)   
  \end{align*}
  By \cref{1pr:3.1},  $R\alpha_{-*}\mathcal{O}_{\mathfrak{Y}}=\mathcal{O}_{\mathfrak{Z}_-}$ and  by \cref{2pr:3.1}
  \begin{align}
    e_{m-r}f_r&=R\theta_*(\mathcal{L}_2'^{m-r}\mathcal{L}_2^{r-1}\mathcal{O}_{\mathfrak{Y}})[1] \nonumber \\
    \label{eq:7.1} &=R\theta_*(\mathcal{L}_1'^{m-r}\mathcal{L}_1^{r-1}\mathcal{O}((m-1)\Delta_{\mathfrak{Y}}))[1].
  \end{align}

\subsection{$a_{m,r}^{k}$ and $\mathfrak{A}_{m,r}^k$}
  We recall the short exact sequence 
\begin{equation}
    \label{eq:7.2}
    0\to \mathcal{O}_{\mathfrak{Y}}(-\Delta_{\mathfrak{Y}})\to \mathcal{O}_{\mathfrak{Y}} \to \mathcal{O}_{\Delta_{\mathfrak{Y}}}\to 0.
  \end{equation}
  \begin{definition}
    For any integer $r$, we define
    $$a_{m,r}^{k}=R\theta_{*}(\mathcal{L}_{1}'^{m-r}\mathcal{L}_{1}^{r-1}\mathcal{O}(k\Delta_{\mathfrak{Y}}))[1].$$
     We define the exact traingles
  \begin{equation}
    \label{eq:7.3}
    \mathfrak{A}_{m,r}^k :a_{m,r}^{k-1}\to a_{m,r}^{k} \to R\Delta_*(R(p_n\circ q_n)_*\mathcal{L}_{1}^{m-1-k}\mathcal{L}_{2}^k)[1] \to a_{m,r}^{k-1}[1]
  \end{equation}
  by applying the functor $R\theta_{*}(-\otimes \mathcal{L}_{1}'^{m-r}\mathcal{L}_{1}^{r-1}\mathcal{O}(k\Delta_{\mathfrak{Y}}))[1]$ to \cref{eq:7.2}.
  \end{definition}
  By \cref{eq:7.1}, $e_{m-r}f_r\cong a_{m,r}^{m-1}$.  Moreover,  when $m<0$, by \cref{1:rmk}, $R(p_n\circ q_n)_*\mathcal{L}_{1}^{-1}\mathcal{L}_{2}^m=0$ and $e_{m-r}f_r\cong a_{m,r}^{m-1}\cong a_{m,r}^m$.

\subsection{$\mathfrak{B}_{m,r}'^0$ and $\mathfrak{C}_{m,r}'^{0}$}
By \cref{canonical},  $R\alpha_{+*}\mathcal{O}_{\Delta_{\mathfrak{Y}}}=\mathcal{O}_{W_0\cap W_1}$. Taking the $R\alpha_{+*}$ functor to the short exact sequence \cref{eq:7.2}, we have the short exact sequence 
    \begin{equation*}
      0\to R\alpha_{+*}\mathcal{O}(-\Delta_{\mathfrak{Y}})\to \mathcal{O}_{W_0}\to \mathcal{O}_{W_0\cap W_1}\to 0.      \end{equation*}
    Recalling the short exact sequence in \cref{eq:Zn}:
    $$0\to \mathcal{L}\otimes \mathcal{O}_{W_1}\to \mathcal{O}_{W_1}\to \mathcal{O}_{W_0\cap W_1}\to 0.$$
    We have the following commutative diagram:
      \begin{equation*}
        \begin{tikzcd}
          & & 0\ar{d} & 0\ar{d} \\
          & &  R\alpha_{+*}\mathcal{O}(-\Delta_{\mathfrak{Y}}) \ar{d}\ar{r}{\cong} &  R\alpha_{+*}\mathcal{O}(-\Delta_{\mathfrak{Y}}) \ar{d} \\
          0 \ar{r} & \mathcal{L}\otimes \mathcal{O}_{W_1} \ar{r} \ar{d}{\cong} & \mathcal{O}_{\mathfrak{Z}_+} \ar{r} \ar{d} & \mathcal{O}_{W_0} \ar{r} \ar{d} & 0 \\
          0 \ar{r} & \mathcal{L}\otimes \mathcal{O}_{W_1}\ar{r} & \mathcal{O}_{W_1} \ar{r}\ar{d} & \mathcal{O}_{W_0\cap W_1} \ar{r}\ar{d} & 0 \\
          & & 0 & 0
        \end{tikzcd}
      \end{equation*}
      where all rows and columns are short exact sequences. Thus we have short exact sequences:
      \begin{align}
         \label{eq:7.4}
        0\to R\alpha_{+*}\mathcal{O}(-\Delta_{\mathfrak{Y}})\to \mathcal{O}_{\mathfrak{Z}_+} \to \mathcal{O}_{W_1}\to 0 \\
          \label{eq:7.6}
0\to \mathcal{L}_1\mathcal{O}_{S^{[n,n+1]}}\to \mathcal{O}_{\mathfrak{Z}_{+}} \to  \mathcal{O}_{W_0}\to 0.
      \end{align}
Taking the functor $R\beta_{+*}(-\otimes \mathcal{L}_{1}'^{m-r}\mathcal{L}_{1}^{r-1})[1]$ to \cref{eq:7.4} and \cref{eq:7.6} respectively, we get the exact triangles
\begin{align}
  \label{eq:7.5}
 & \mathfrak{C}_{m,r}'^{0}:a_{m,r}^{-1}\to f_{r}e_{m-r} \to R\Delta_*(Rp_{n*}\mathcal{L}^{m-1})[1] \to a_{m,r}^{-1}[1]  \\
 \label{eq:7.7}
  &  \mathfrak{B}_{m,r}'^{0}:  R\Delta_*(Rp_{n*}\mathcal{L}^{m})[1] \to f_{r}e_{m-r} \to a_{m,r}^0 \to R\Delta_*(Rp_{n*}\mathcal{L}^{m})[2].
\end{align}

\subsection{$\mathfrak{B}_{m,r}^k$ and $\mathfrak{C}_{m,r}^k$}

\begin{definition}
  When $m>0$, we define $b_{m.r}^k\in D^{b}(S^{[n]}\times S^{[n]}\times S\times S)$ and natural transforms $\mathfrak{B}_{m,r}^k:b_{m,r}^{k}\to b_{m,r}^{k+1}$ by

  \begin{minipage}{0.5\linewidth}
     \begin{equation}
    b_{m,r}^k:= 
    \begin{cases}
      f_{r}e_{m-r} & k=0\\
      a_{m,r}^{k-1} & 1\leq k\leq m
    \end{cases}
  \end{equation}
\end{minipage}
\begin{minipage}{0.5\linewidth}
    \begin{equation}
    \mathfrak{B}_{m,r}^k:=
    \begin{cases}
      \mathfrak{B}_{m,r}'^{0} & k=0\\
      \mathfrak{A}_{m,r}^{k}& 1\leq k\leq m-1.
    \end{cases}
  \end{equation}
\end{minipage}

We have $b_{m,r}^{0}=f_re_{m-r}$ and $b_{m,r}^{m}=e_{m-r}f_r$.

When $m<0$, we define $c_{m,r}^k\in D^{b}(S^{[n]}\times S^{[n]}\times S\times S)$ and $\mathfrak{C}_{m,r}^k:c_{m,r}^{k-1}\to c_{m,r}^{k}$ by

\begin{minipage}{0.5\linewidth}
  \begin{equation}
    c_{m,r}^k=
    \begin{cases}
      f_{r}e_{m-r} & k=0\\
      a_{m,r}^k & m\leq k\leq -1
    \end{cases}
  \end{equation}
\end{minipage}
\begin{minipage}{0.5\linewidth}
   \begin{equation}
    \label{eq:7.8}
    \mathfrak{C}_{m,r}^k=
    \begin{cases}
      \mathfrak{C}_{m,r}'^{0} & k=0 \\
     \mathfrak{A}_{m,r}^k & m+1 \leq k \leq -1
    \end{cases}
  \end{equation}
\end{minipage}
We have $f_re_{m-r}=c_{m,r}^{0}$ and $e_{m-r}f_r=c_{m,r}^{m}$.
\end{definition}

  \begin{proof}[Proof of \cref{thm:7.1}]When $m>0$, we only need to prove that the cone of $\mathfrak{B}_{m,r}^k$ is $R\Delta_*(h_{m,k}^+)$ and the cone of  $\mathfrak{C}_{m,r}^k$ is $R\Delta_*(h_{m,k}^-)$. It follows from \cref{eq:7.3}, \cref{eq:7.5} and \cref{eq:7.7}.

     When $m=0$,
  \begin{align*}
    e_{-r}f_r=R\beta_{+*}(R\alpha_{+*}(\mathcal{L}_1'^{-r}\mathcal{L}_1^{r-1}\mathcal{O}(-\Delta_{\mathfrak{Y}}))).
  \end{align*}
  We have the short exact sequence:
  \begin{equation}
    \label{eq:7.12}
      0\to \mathcal{L}_1'^{-r}\mathcal{L}_1^{r-1}\mathcal{O}_{\mathfrak{Z}_+}\to \mathcal{L}_1'^{-r}\mathcal{L}_1^{r-1}(\mathcal{O}_{W_0}\oplus\mathcal{O}_{W_1})\to \mathcal{L}^{-1}\mathcal{O}_{W_0\cap W_1}\to 0.
  \end{equation}

  and
  \begin{align*}
    R\beta_{+*}(\mathcal{L}^{-1}\mathcal{O}_{W_0\cap W_1})&=\{R\beta_{+*}(\mathcal{O}_{W_1}\to \mathcal{L}^{-1}\mathcal{O}_{W_1})\} \\
                                          &=R\Delta_{*}\{\mathcal{O}_{S^{[n]}\times S}\to \mathcal{O}_{S^{[n]}\times S}\} \text{ by \cref{L1} }\\
                                          &=0
  \end{align*}

  By \cref{eq:7.4} and \cref{eq:7.12}, we have isomorphisms
  \begin{align*}
   & f_{r}e_{-r}\cong R\beta_{+*}\mathcal{L}_1'^{-r}\mathcal{L}_1^{r-1}(\mathcal{O}_{W_0})[1]\oplus\mathcal{O}_{\Delta}[1] \\
   & e_{-r}f_{r}\cong R\beta_{+*}\mathcal{L}_1'^{-r}\mathcal{L}_1^{r-1}(\mathcal{O}_{W_0})[1]
  \end{align*}
  \end{proof}

\section{Extension Classes between $h_{m,k}^{\pm}$}
\label{sec:7}
When $m>0$, the composition of $\mathfrak{B}_{m,r}^k$ and $\mathfrak{B}_{m,r}^{k+1}$ in \cref{thm:7.1} induces an extension of $R\Delta_*h_{m,k}^{+}$ and $R\Delta_*h_{m,k+1}^{+}$, i.e. an extension class of
$$Hom(R\Delta_*h_{m,k+1}^{+},R\Delta_*h_{m,k}^{+}[1]).$$
By the Hochschild-Kostant-Rosenberg theorem (\cite{MR1390671}, also \cref{rmk7.4}),
$$Hom(R\Delta_*h_{m,k+1}^{+},R\Delta_*h_{m,k}^{+}[1])=\bigoplus_{j=0}^\infty Hom(h_{m,k+1}^+\otimes \wedge^jT^*(S^{[n]}\times S),h_{m,k}^+[1-j]).$$
When $m<0$, there is also be an extension class of
$$Hom(R\Delta_*h_{m,k}^{-},R\Delta_*h_{m,k-1}^{-}[1])=\bigoplus_{j=0}^\infty Hom(h_{m,k}^{-}\otimes \wedge^jT^*(S^{[n]}\times S),h_{m,k-1}^{-}[1-j]).$$
In this section, we will explicitly computing the above extension classes and prove that
\begin{proposition}
  \label{ext2}
  The extension class in $Hom(R\Delta_*h_{m,k+1}^{+},R\Delta_*h_{m,k}^{+}[1])$ is
  $$(0,\mathfrak{H}_{m,k}^{+},0,0,\cdots)$$
  and  the extension class in $Hom(R\Delta_*h_{m,k}^{-},R\Delta_*h_{m,k-1}^{-}[1])$ is
  $$(0,\mathfrak{H}_{m,k}^-,0,0,\cdots)$$
  where $\mathfrak{H}_{m,k}^{-}$ and $\mathfrak{H}_{m,k}^{-}$ will be defined in \cref{extension}.
\end{proposition}
\cref{ext2} will be proved later in this section.
\subsection{A Splitting Formula and HKR isomorphism}
Let $j:Y\subset X$ be a regular embedding of smooth schemes and $N_{Y/X}$ be the normal bundle. In a first step, we will suppose that $Y$ is given as the zero locus of a regular section $s\in H^0(X,\mathcal{E})$ of a locally free sheaf $\mathcal{E}$ of rank $r$. In this case, the structure sheaf $\mathcal{O}_{Y}$ can be resolved by the Koszul complex:
\begin{equation}
  0\to \wedge^r \mathcal{E}^\vee \to \cdots \to \mathcal{E}^\vee\to \mathcal{O}_X\to j_*\mathcal{O}_Y\to 0,
\end{equation}
with isomorphism given by contraction with $s$. The normal bundle $N_{Y/X}$ is $\mathcal{E}|_Y$.
\begin{proposition}[Proposition 11.1 of \cite{MR2244106}]
  \label{prop:11.1}
  There exists a canonical isomorphism:
  $$Lj^*Rj_*\mathcal{O}_Y\cong \bigoplus_k \wedge^k N_{X/Y}^\vee[k].$$
\end{proposition}

A stronger version of \cref{prop:11.1} was studied in \cite{MR4003476}:
\begin{theorem}[Theorem 1.8, Lemma 3.1 of \cite{MR4003476}]
  \label{thm1.8}
  Let $j:Y\to X$ be a closed embedding of smooth varieties. A choice of splitting of
\begin{equation}
  \label{eq:HKR}
0\to T_{Y}\to T_X|_Y\to N_{X/Y}\to 0  
\end{equation}
  determines an isomorphism:
  $$Lj^*Rj_*(\mathcal{O}_Y)\cong \bigoplus_k \wedge^k N_{X/Y}^\vee[k].$$
\end{theorem}

\begin{remark}
  \label{rmk7.4}
  Let $X$ be a smooth variety and $\Delta_X:X\to X\times X$ the diagonal embedding. Then $T_{X\times X}|_X=T_X\oplus T_X$ has a canonical splitting and thus
  $$L\Delta_X^*R\Delta_{X*}(\mathcal{O}_X)=\bigoplus_{k=0}^{dim(X)}\wedge^k T^*X[k]$$
  which is the HKR isomorphism in \cite{MR1390671}.
\end{remark}

\subsection{A morphism $s:\mathcal{L}_{1}^{-1}\mathcal{L}_{2}\to (p_n\circ q_n)^{*}T(S^{n}\times S)$}
In this subsection, we first construct two canonical morphisms:
\begin{align}
    \label{2eq:2.6}
  s_1:\mathcal{L}_1^{-1}\mathcal{L}_2\to  (p_n\circ q_n)^*TS^{[n]}\quad  s_2:\mathcal{L}_1^{-1}\mathcal{L}_2\to (p_n\circ q_n)^*(TS)
\end{align}

Let $\pi:S^{[n-1,n,n+1]}\times S\to S^{[n-1,n,n+1]}$ be the projection morphism and $\Gamma$ be the graph of the projection map to $S$. Then there exists short exact sequences:
$$0\to \mathcal{I}_{n+1}\to \mathcal{I}_n\to \pi^*\mathcal{L}_{1}\otimes \mathcal{O}_{\Gamma}\to 0\quad 0\to \mathcal{I}_{n}\to \mathcal{I}_{n-1}\to \pi^*\mathcal{L}_{2}\otimes \mathcal{O}_{\Gamma}\to 0$$
which induces two global sections:
$$r_1:\mathcal{L}_1^{-1}\to \mathcal{H}om_{\pi}(\mathcal{I}_n,\mathcal{O}_{\Gamma}), \quad r_2:\mathcal{L}_2\to \mathcal{E}xt^1_\pi(\mathcal{O}_{\Gamma},\mathcal{I}_{n})$$
Let $\text{proj}_n:S^{[n]}\times S\to S^{[n]}$ be the projection map and $\Gamma_n$ be the graph of the projection map from $S^{[n]}\times S\to S$. Then 
\begin{align*}
  (p_n\circ q_n)^*TS^{[n]}&=(p_n\circ q_n)^*\mathcal{E}xt_{\text{proj}_n}^1(\mathcal{I}_n,\mathcal{I}_n)\\
                          &=\mathcal{E}xt^1_\pi(\mathcal{I}_n,\mathcal{I}_n)\\
  (p_n\circ q_n)^*TS&=\mathcal{E}xt^1_\pi(\mathcal{O}_\Gamma,\mathcal{O}_\Gamma)
\end{align*}
by the flat base change theorem. With the composition of the following two natural homomorphisms:
\begin{align*}
  \mathcal{E}xt^1_\pi(\mathcal{O}_\Gamma,\mathcal{I}_n)\otimes \mathcal{H}om_\pi(\mathcal{I}_n, \mathcal{O}_\Gamma) \to \mathcal{E}xt^1_\pi(\mathcal{I}_n,\mathcal{I}_n)=(p_n\circ q_n)^*TS^{[n]} \\
  \mathcal{H}om_\pi(\mathcal{I}_n,\mathcal{O}_\Gamma)\otimes \mathcal{E}xt^1_\pi(\mathcal{O}_\Gamma,\mathcal{I}_n)\to \mathcal{E}xt^1_\pi(\mathcal{O}_\Gamma,\mathcal{O}_\Gamma)=(p_n\circ q_n)^*TS,
\end{align*}
we get the two canonical morphisms in \cref{2eq:2.6}
\begin{align*}
  s_1:\mathcal{L}_1^{-1}\mathcal{L}_2\to  (p_n\circ q_n)^*TS^{[n]}\quad  s_2:\mathcal{L}_1^{-1}\mathcal{L}_2\to (p_n\circ q_n)^*(TS).
\end{align*}
Let $$s=(r_1,r_2):\mathcal{L}_1^{-1}\mathcal{L}_2\to (p_n\circ q_n)^*T(S^{[n]}\times S).$$

Another way of inducing the morphism is restricting $\theta:\mathfrak{Y}\to S^{[n]}\times S^{[n]}\times S\times S$ to the diagonal of $S^{[n]}\times S^{[n]}\times S\times S$. We get the Cartesian diagram:
\begin{equation}
  \label{2eq:3.5}
  \begin{tikzcd}
    S^{[n-1,n,n+1]} \ar{d}{p_n\circ q_n}\ar{r}{\Delta_{\mathfrak{Y}}} &\mathfrak{Y} \ar{d}{\theta}\\
    S^{[n]}\times S \ar{r}{\Delta} & S^{[n]}\times S^{[n]}\times S \times S
  \end{tikzcd}
\end{equation}
and the normal bundle of $S^{[n]}\times S$ in  $S^{[n]}\times S^{[n]}\times S\times S$ is $T(S^{[n]}\times S)$. Thus the diagram \cref{2eq:3.5} induces a morphism
\begin{equation}
  s':N_{\mathfrak{Y}/\Delta_{\mathfrak{Y}}}=\mathcal{L}_1^{-1}\mathcal{L}_2\to (p_n\circ q_n)_*T(S^{[n]}\times S).
\end{equation}

\begin{proposition}
   The morphism $s'$ coincides with $s$ in \cref{2eq:2.6}.
\end{proposition}
\begin{proof}
  A closed point $t$ on $S^{[n-1,n,n+1]}$ corresponds to two short exact sequences of coherent sheaves on $S$:
  \begin{align*}
    0\to \mathcal{I}_{n}\xrightarrow{i_n} \mathcal{I}_{n-1}\xrightarrow{j_n} k_x \to 0 & \quad
    0\to \mathcal{I}_{n+1}\xrightarrow{i_{n+1}} \mathcal{I}_{n}\xrightarrow{j_{n+1}} k_x \to 0.
  \end{align*}
  which induces $r_1\in Hom(\mathcal{I}_n,k_x)$ and $r_2\in Ext^1(k_x,\mathcal{I}_n)$. Let $V$ be the vector space of pairs $\{(w_0,w_1)\in Ext^1(\mathcal{I}_{n-1},\mathcal{I}_{n-1})\oplus Ext^1(\mathcal{I}_{n},\mathcal{I}_n)\}$ such that  $w_0,w_1$ map to the same element of  $Ext^1(\mathcal{I}_n,\mathcal{I}_{n-1})$.   By the proof of Proposition 5.28 of \cite{neguct2018hecke}, elements in $V$ are in 1-1 correspondence with the commutative diagrams of short exact sequences:
   \begin{equation*}
    \begin{tikzcd}
       & 0 \ar{d} & 0\ar{d} & 0\ar{d} & \\
       0 \ar{r} & \mathcal{I}_{n} \ar{r} \ar{d}{i_{n}} & \mathcal{A} \ar{r}{} \ar{d} & \mathcal{I}_{n} \ar{r} \ar{d}{i_{n}} & 0 \\
       0 \ar{r} & \mathcal{I}_{n-1} \ar{r} \ar{d}{j_n} & \mathcal{B} \ar{r} \ar{d} &\mathcal{I}_{n-1}\ar{r} \ar{d}{j_n}& 0 \\
      0\ar{r} & k_x \ar{r}\ar{d} & \mathcal{C} \ar{r} \ar{d} & k_x \ar{r} \ar{d} & 0 \\
      & 0 & 0& 0&
    \end{tikzcd}
  \end{equation*}
  and thus induce a morphism $dp_S^1:V\to Ext^1(k_x,k_x)$. Let $V'$ be the vector space of pairs
  $$\{(w_1',w_2)\in Ext^1(\mathcal{I}_{n},\mathcal{I}_n)\oplus Ext^1(\mathcal{I}_{n+1},\mathcal{I}_{n+1})\}$$
  auch that $w_1',w_2$  map to the same element of $Ext^1(\mathcal{I}_{n},\mathcal{I}_{n+1})$. Then elements in $V'$  are in 1-1 correspondence with the commutative diagrams of short exact sequences:
    \begin{equation*}
    \begin{tikzcd}
       & 0 \ar{d} & 0\ar{d} & 0\ar{d} & \\
       0 \ar{r} & \mathcal{I}_{n+1} \ar{r} \ar{d} & \mathcal{A}' \ar{r}{} \ar{d} & \mathcal{I}_{n+1} \ar{r} \ar{d} & 0 \\
       0 \ar{r} & \mathcal{I}_{n} \ar{r} \ar{d}{} & \mathcal{B}' \ar{r} \ar{d} &\mathcal{I}_n \ar{r} \ar{d}& 0 \\
      0\ar{r} & k_x \ar{r}\ar{d} & \mathcal{C}' \ar{r} \ar{d} & k_x \ar{r} \ar{d} & 0 \\
      & 0 & 0& 0&
    \end{tikzcd}
  \end{equation*}
  and thus also induces a morphism $dp_S^2:V'\to Ext^1(k_x,k_x)$.

 Consider the natural morphism:
\begin{align*}
  \hat{s}_1: Hom(\mathcal{I}_n,k_x)\otimes Ext^1(k_x,\mathcal{I}_n)\to Ext^1(\mathcal{I}_n,\mathcal{I}_n) \\
  \hat{s}_2: Hom(\mathcal{I}_n,k_x)\otimes Ext^1(k_x,\mathcal{I}_n)\to Ext^1(k_x,k_x).
\end{align*}
and the short exact sequences
\begin{align*}
  Hom(\mathcal{I}_n,k_x)\xrightarrow{\hat{s}_1(-\otimes r_2)} Ext^1(\mathcal{I}_n,\mathcal{I}_n)\to Ext^1(\mathcal{I}_n,\mathcal{I}_{n-1}) \\
  Ext^{1}(k_x,\mathcal{I}_n)\xrightarrow{\hat{s}_1(r_1\otimes -)}Ext^1(\mathcal{I}_n,\mathcal{I}_n)\to Ext^1(\mathcal{I}_{n+1},\mathcal{I}_n).
\end{align*}
For any element $u\in Hom(\mathcal{I}_n,k_x)$, $(\hat{s}_1(u\otimes r_2),0)\in V$. Moreover, by diagram chasing (we left it to interested readers)
\begin{equation}
  \label{eq:chasing1}
  dp_S^1(\hat{s}_1(u\otimes r_2),0)=\hat{s}_2(u\otimes r_2).
\end{equation}
Simlilarly, for any element $v\in Ext^1(k_x,\mathcal{I}_n)$, $(0,\hat{s}_1(r_1\otimes v))\in V'$ and
\begin{equation}
  \label{eq:chasing2}
  dp_S^2(0,\hat{s}_1(r_1\otimes v))=\hat{s}_2(r_1\otimes v).
\end{equation}

 By Proposition 5.28 of \cite{neguct2018hecke}, the tangent space $T_{t}\mathfrak{Y}$ is the space of $(w_0,w_1,w_1',w_2,u,v)$ in
  $$Ext^1(\mathcal{I}_{n-1},\mathcal{I}_{n-1})\oplus Ext^{1}(\mathcal{I}_n,\mathcal{I}_n)\oplus  Ext^{1}(\mathcal{I}_n,\mathcal{I}_n)\oplus Ext^{1}(\mathcal{I}_{n+1},\mathcal{I}_{n+1})\oplus Ext^1(k_x,k_x)\oplus Ext^1(k_x,k_x)$$
  such that 
\begin{enumerate}
\item $(w_0,w_1)$ and $(w_0,w_1')$ are in $A$.
\item $(w_1,w_2)$ and $(w_1',w_2)$ are in $A'$.
\item $dp_S^1(w_1,w_0)=dp_S^2(w_2,w_1')=u$,
\item $dp_S^1(w_1',w_0)=dp_S^2(w_2,w_1)=v$.
\end{enumerate}
If $(w_0,w_1,w_1',w_2,u,v)$ is in the tangent space, then so is
$$\frac{(2w_0,w_1+w_1',w_1+w_1',2w_2,u+v,u+v)}{2} \text{ and } (0,w_1'-w_1,w_1'-w_1,0,u-v,v-u).$$
Hence the tangent space $T_t\mathfrak{Y}$ decompose into a direct sum of two subspaces: the subspace $w_1=w_1'$ which is $T_tS^{[n-1,n,n+1]}$ and the subspace $N$ which consists of elements $(w_1,u)\in Ext^{1}(\mathcal{I}_n,\mathcal{I}_n)\oplus Ext^{1}(k_x,k_x)$ such that
\begin{enumerate}
\item $w_1$ maps to $0$ in $Ext^{1}(\mathcal{I}_{n}, \mathcal{I}_{n-1})$ and $Ext^{1}(\mathcal{I}_{n+1}, \mathcal{I}_{n}),$
\item $dp_S^1(w_1,0)=dp_S^2(0,w_1)=u.$
\end{enumerate}

The image of $s$ is $(\hat{s}_1(r_1\otimes r_2),\hat{s}_2(r_1\otimes r_2))$ and we only need to prove that it is in $N$. It follows from the fact that
$$dp_S^1(\hat{s}_1(r_1\otimes r_2),0)=dp_S^2(0,\hat{s}_1(r_1\otimes r_2))=\hat{s}_2(r_1\otimes r_2)$$
by \cref{eq:chasing1}, \cref{eq:chasing2}.

\end{proof}

\subsection{The definition of $\mathfrak{H}_{m,k}^{\pm}$}
Consider the dual of $s$
$$s^\vee:(p_n\circ q_n)^*T^*(S^{[n]}\times S)\to \mathcal{L}_1\mathcal{L}_2^{-1}.$$  Since $L(p_n\circ q_n)^*$ is the left adjoint functor of $R(p_n\circ q_n)_*$, $s^\vee$ induces another morphism
\begin{equation*}
\mathfrak{h}:T^*(S^{[n]}\times S)\to R(p_n\circ q_n)_*(\mathcal{L}_1\mathcal{L}_2^{-1}).
\end{equation*}

The identity
$$\mathcal{L}_1^{m-1-k}\mathcal{L}_2^k\otimes \mathcal{L}_1\mathcal{L}_2^{-1}\cong \mathcal{L}_1^{m-k}\mathcal{L}_2^{k-1}$$
induces morphisms
\begin{equation}
  \label{onlyonce}
  R(p_n\circ q_n)_*\mathcal{L}_1^{m-1-k}\mathcal{L}_2^k\otimes^{L}R(p_n\circ q_n)_*\mathcal{L}_1\mathcal{L}_2^{-1}\to R(p_n\circ q_n)_*\mathcal{L}_1^{m-k}\mathcal{L}_2^{k-1}
\end{equation}
and we define
\begin{equation}
  \label{eq:test1}
\tau_{m}^k:R(p_n\circ q_n)_*\mathcal{L}_1^{m-1-k}\mathcal{L}_2^k\otimes^{L} T^*(S^{[n]}\times S)\to R(p_n\circ q_n)_*\mathcal{L}_1^{m-k}\mathcal{L}_2^{k-1}.
\end{equation}
to be the composition of $\mathfrak{h}$ and \cref{onlyonce}.

Still by \cref{lem:4.5},
$$Rq_{n*}(\mathcal{L}_1^{m-1})=\{\mathcal{L}^m\to \mathcal{L}^{m-1}\}$$
is the cone of $\mathcal{L}^m$ to $\mathcal{L}^{m-1}$. Hence we have exact triangles
\begin{equation}
  Rp_{n*}\mathcal{L}^m\to Rp_{n*}\mathcal{L}^{m-1}\xrightarrow{\mu_m} R(p_n\circ q_n)_*(\mathcal{L}_1^{m-1}) \xrightarrow{\nu_m}Rp_{n*}\mathcal{L}^m[1].
\end{equation}
\begin{definition}
  \label{extension}
  When $m>k\geq 0$, we define
  $$\mathfrak{H}_{m,k}^+:h_{m,k+1}^+\otimes T^*(S^{[n]}\times S)\to  h_{m,k}^+$$
  by
  \begin{equation}
    \mathfrak{H}_{m,k}^+=
    \begin{cases}
      \tau_{m,k} & k>0 \\
      \nu_m\circ \tau_{m,0}& k=0
    \end{cases}
  \end{equation}
  and when $m<k\leq 0$, we define
  $$\mathfrak{H}_{m,k}^-:h_{m,k}^-\otimes T^*(S^{[n]}\times S)\to  h_{m,k-1}^+$$
  by
   \begin{equation}
    \mathfrak{H}_{m,k}^+=
    \begin{cases}
      \tau_{m,k} & k<0 \\
      \tau_{m,0}\circ \mu_m & k=0
    \end{cases}
  \end{equation}
\end{definition}

\begin{proof}[Proof of \cref{ext2}]
  We only compute the postive part of the extension and only consider the $k>0$ and the case $k=0$ follows from the \cref{eq:7.6}. The extension class is induced from the short exact sequence on $\mathfrak{Y}$:
  $$0\to \mathcal{L}_{1}^{m-k-1}\mathcal{L}_2^{k}\mathcal{O}_{\Delta_{\mathfrak{Y}}}\to \mathcal{L}_{1}^{m-k-2}\mathcal{L}_2^{k+1}\mathcal{O}_{2\Delta_{\mathfrak{Y}}}\to \mathcal{L}_1^{m-k-2}\mathcal{L}_2^{k+1}\mathcal{O}_{\Delta_{\mathfrak{Y}}}\to 0$$
  where $\mathcal{O}_{2\Delta_{\mathfrak{Y}}}$ is the cokernel of $\mathcal{O}(-2\Delta_{\mathfrak{Y}})\to \mathcal{O}_{\mathfrak{Y}}$.   It induces an extension class in
  \begin{align*}
    &Ext^1_{\mathfrak{Y}}( \mathcal{L}_1^{m-k-2}\mathcal{L}_2^{k+1}\mathcal{O}_{\Delta_{\mathfrak{Y}}},\mathcal{L}_{1}^{m-k-1}\mathcal{L}_2^{k}\mathcal{O}_{\Delta_{\mathfrak{Y}}}) \\
    =&Ext^1_{\Delta_{\mathfrak{Y}}}(\mathcal{L}_1^{m-k-2}\mathcal{L}_2^{k+1},\mathcal{L}_{1}^{m-k-1}\mathcal{L}_2^{k})\oplus Hom_{\Delta{\mathfrak{Y}}}(\mathcal{L}_{1}^{m-k-1}\mathcal{L}_2^{k},\mathcal{L}_{1}^{m-k-1}\mathcal{L}_2^{k}).
  \end{align*}
   by \cref{thm1.8} and (3) of  \cref{2pr:3.1}. Still by \cref{thm1.8}, the extension class is $(0,id)$ and thus would map to $(0,\mathfrak{H}_{m,k}^{+},0,0,\cdots)$ by the functorial property.
\end{proof}
\appendix

 \section{ Exterior Powers and  Formal Series}
 \label{sec:b}
\begin{definition}
   Let $V$ be a locally free sheaf over $X$, we define the exterior powers of $V$ by
 $$\wedge^{\bullet}(xV)=\sum_{i=0}^\infty(-x)^i[\wedge^iV], \quad \wedge^\bullet(-xV)=\sum_{i=0}^\infty x^i[S^iV]$$
 as elements in $K(X)[[x]]$.
\end{definition}
\begin{definition}
  \label{sku}
  For a two term complex of locally free sheaves
  $$U:=\{W\xrightarrow{\mathfrak{u}}V\}$$
  we define
  $$\wedge^\bullet(-xU)=\wedge^\bullet(xW)\wedge^\bullet(-xV)\qquad \wedge^\bullet(xU)=\wedge^\bullet(xV)\wedge^\bullet(-xW).n$$
  $[S^k(U)]$ and $[\wedge^k(U)]$ are the $x^k$ coefficients of $\wedge^\bullet(-xU)$ and $\wedge^\bullet(xU)$ respectively.
\end{definition}
\begin{lemma}[Exercise \RN{2}.5.16 of \cite{hartshorne2013algebraic}]For a two term complex of locally free sheaves
  $$U:=\{W\xrightarrow{\mathfrak{u}}V\},$$
  we have
  \begin{equation}
  \label{eq:=1}
\wedge^{\bullet}(xU)\wedge^\bullet(-xU)=1.  
\end{equation}
\end{lemma}

\begin{definition}[Section 3.8 of \cite{negut2017shuffle}]
  \label{appendix1}
  For any non-negative integer $m$, we define $h_{m}^\pm\in K(S^{[n]}\times S)$ 
  \begin{align}
      h_m^+:=[\omega_S]\sum_{i=0}^m(-1)^{m-i}[\omega_S^{-i}][S^i\mathcal{I}_n][\wedge^{m-i}\mathcal{I}_n] \\
    h_{m}^{-}:=\sum_{i=0}^{m}(-1)^i[\omega_S^{i}][\wedge^{i}\mathcal{I}_n^\vee][S^{m-i}\mathcal{I}_n^\vee] 
  \end{align}
  where we abuse the notation to denote
  $$\mathcal{I}_n:=\{W_n\xrightarrow{s}V_n\}$$
  in the short exact sequence \cref{eq:resolve}.
\end{definition}

\begin{lemma}
  \cref{appendix1} and \cref{4eq:lem2.11} are equivalent.
\end{lemma}
\begin{proof}
   First note that
  \begin{equation}
    \label{eq:+1}
    \sum_{i=0}^m[\wedge^i\mathcal{I}_n][S^{m-i}\mathcal{I}_n]=0
  \end{equation}
  by \cref{eq:=1}. Thus
  \begin{align*}
    &[\omega_S]\sum_{i=0}^m(-1)^{m-i}[\omega_S^{-i}][S^i\mathcal{I}_n][\wedge^{m-i}\mathcal{I}_n]   \\
         &=[\omega_S]\sum_{i=1}^m(-1)^{m-i}([\omega_S^{-i}]-1)[S^i\mathcal{I}_n][\wedge^{m-i}\mathcal{I}_n] & \text{ by } \cref{eq:+1} \\
         &=([\omega_S]-1)\sum_{i=1}^{m}(-1)^{m-i}[S^{i}\mathcal{I}_n][\wedge^{m-i}\mathcal{I}_n]\sum_{j=1}^{i}[\omega_S^{-j+1}]  \\
         &=([\omega_S]-1)\sum_{j=1}^{m}[\omega_S^{-j+1}]\sum_{i=j}^{m}(-1)^{m-i}[S^{i}\mathcal{I}_n][\wedge^{m-i}\mathcal{I}_n]  \\
         &=(1-[\omega_S])\sum_{j=0}^{m-1}[\omega_S^{-j}]\sum_{i=0}^{j}(-1)^{i}[S^{m-i}\mathcal{I}_n][\wedge^{i}\mathcal{I}_n]. & \text{ by } \cref{eq:+1}
  \end{align*}
  Hence $h_{m}^{+}$ are equivalent in two defintions.
  
  Then we note that
  \begin{equation}
    \label{eq:+2}
    \sum_{i=0}^{m}(-1)^i[S^{i}\mathcal{I}_n^\vee][\wedge^{m-i}\mathcal{I}_n^\vee]=0
  \end{equation}
  by \cref{eq:=1}. Thus
  \begin{align*}
    &\sum_{i=0}^{m}(-1)^i[\omega_S^{i}][\wedge^{i}\mathcal{I}_n^\vee][S^{m-i}\mathcal{I}_n^\vee] \\
             &=\sum_{i=1}^m(-1)^i([\omega_S^{i}]-1)[\wedge^i\mathcal{I}_n^\vee][S^{m-i}\mathcal{I}_n^\vee] & \text{ by } \cref{eq:+2}  \\
             &=([\omega_S]-1)\sum_{i=1}^m(-1)^i[\wedge^i\mathcal{I}_n^\vee][S^{m-i}\mathcal{I}_n^\vee]\sum_{j=1}^i[\omega_S^{j-1}]  \\
             &=([\omega_S]-1)\sum_{j=1}^{m}[\omega_S^{j-1}]\sum_{i=j}^m(-1)^i[\wedge^i\mathcal{I}_n^\vee][S^{m-i}\mathcal{I}_n^\vee]   \\
    &=(1-[\omega_S])\sum_{j=0}^{m-1}(-1)^{j}[\omega_S^{j}]\sum_{i=0}^j(-1)^i[\wedge^{m-i}\mathcal{I}_n^\vee][S^i\mathcal{I}_n^\vee]  & \text{ by } \cref{eq:+2} 
  \end{align*}
  Hence $h_{m}^{-}$ are equivalent in two defintions.
\end{proof}

\bibliography{elliptic-hall-algebra1}
\bibliographystyle{plain}

\end{document}